\documentclass{amsart}

\usepackage{amsthm, amsfonts, amssymb, amsmath, latexsym, enumerate, times}
\usepackage[latin1]{inputenc}
\usepackage{xypic}
\usepackage{mathrsfs}
\usepackage{mathtools}
\usepackage{letltxmacro}
\LetLtxMacro\orgvdots\vdots
\LetLtxMacro\orgddots\ddots

\makeatletter
\DeclareRobustCommand\vdots{%
	\mathpalette\@vdots{}%
}
\newcommand*{\@vdots}[2]{%
	\sbox0{$#1\cdotp\cdotp\cdotp\m@th$}%
	\sbox2{$#1.\m@th$}%
	\vbox{%
		\dimen@=\wd0 %
		\advance\dimen@ -3\ht2 %
		\kern.5\dimen@
		\dimen@=\wd2 %
		\advance\dimen@ -\ht2 %
		\dimen2=\wd0 %
		\advance\dimen2 -\dimen@
		\vbox to \dimen2{%
			\offinterlineskip
			\copy2 \vfill\copy2 \vfill\copy2 %
		}%
	}%
}
\DeclareRobustCommand\ddots{%
	\mathinner{%
		\mathpalette\@ddots{}%
		\mkern\thinmuskip
	}%
}
\newcommand*{\@ddots}[2]{%
	\sbox0{$#1\cdotp\cdotp\cdotp\m@th$}%
	\sbox2{$#1.\m@th$}%
	\vbox{%
		\dimen@=\wd0 %
		\advance\dimen@ -3\ht2 %
		\kern.5\dimen@
		\dimen@=\wd2 %
		\advance\dimen@ -\ht2 %
		\dimen2=\wd0 %
		\advance\dimen2 -\dimen@
		\vbox to \dimen2{%
			\offinterlineskip
			\hbox{$#1\mathpunct{.}\m@th$}%
			\vfill
			\hbox{$#1\mathpunct{\kern\wd2}\mathpunct{.}\m@th$}%
			\vfill
			\hbox{$#1\mathpunct{\kern\wd2}\mathpunct{\kern\wd2}\mathpunct{.}\m@th$}%
		}%
	}%
}
\makeatother
\usepackage[normalem]{ulem}
\usepackage{array}
\usepackage{paralist}
\usepackage{xcolor}

\newtheorem{theorem}{Theorem}
\newtheorem{lemma}[theorem]{Lemma}

\newtheorem{corollary}[theorem]{Corollary}
\newtheorem{proposition}[theorem]{Proposition}

\theoremstyle{definition}

\newtheorem{remark}[theorem]{Remark}
\newtheorem{example}[theorem]{Example}

\newcommand{\bbC}{{\mathbb C}}

\newcommand{\bbP}{{\mathbb P}}

\def\geq{\geqslant}
\def\leq{\leqslant}

\begin{document}

\title[{On linear sections of Grassmannians}]{{On linear sections of Grassmannians}}

\author[M. Bolognesi]{Michele Bolognesi}
\address{Univ. Grenoble Alpes, CNRS, IF, 38000 Grenoble, France}
\email{michele.bolognesi@univ-grenoble-alpes.fr}
\author[C. Ciliberto]{Ciro Ciliberto}
\address{Dipartimento di Matematica, Universit\`a di Roma Tor Vergata}
\email{cilibert@axp.mat.uniroma2.it}
\author[A. Verra]{Alessandro Verra}
\address{Dipartimento di Matematica e Fisica, Universit\`a Roma Tre.}
\email{sandro.verra@gmail.com} 

 \thanks{C.C. is member of INdAM-GNSAGA.  The authors thank F. Russo for useful suggestions regarding Theorem \ref {conj:main} and A. Beauville for useful references. }

 

\begin{abstract}  {{ We consider general linear sections $X$ of a Grassmannian $\mathbb G(k,n)$ of $k$--dimensional linear subspaces of $\bbP^n$. We first revisit a classical theorem of Severi and Roth, 
according to which a general $X$ of codimension $c\leq n-k$ is rational. For $c=n-k$ there is a natural birational map $X\dasharrow \mathbb G(k-1,n-1)$, and we describe its inverse and the codimension 2 base scheme of the latter. }Then we focus on smooth sections $X$ of the Grassmannian $\mathbb{G}(1,n)$ which are Fano varieties of index 1 or Calabi--Yau varieties.} \par  
{In the Fano case we  prove that $X$ is birational to a hypersurface of $\bbP^{n-1}$ defined by a Pfaffian of linear forms, and that the union of the lines parametrized by $X$
is a hypersurface of degree $n-1$ in $\bbP^n$; if $n$ is even, $X$ and this hypersurface are rational. }\par
In the Calabi--Yau case, with $n=2k+1 \geq 5$, there exists a Pfaffian hypersurface \(\mathcal{C}\) of degree \(k+1\) in \(\mathbb{P}^{n}\) associated with \(X\), which is      the intersection of the dual of $\mathbb G(1,n)$ with a linear space of dimension $n$. {We provide a birational map between $\mathcal C$ and $\mathbb P^{n-1}$}. More generally, we prove that the general Pfaffian hypersurface of degree $k+1$ in $\bbP^m$ is rational for every $m\geq n$.  \end{abstract}

\maketitle 

\tableofcontents  
\section*{Introduction} 

{The geometry of the Grassmann varieties $\mathbb G(k,n)$ goes back to H. Grassmann (see \cite {De}). The case $k = 1$, that is, line geometry, was then deeply developed by J. Pl\"ucker \cite{Plu}, F. Klein \cite{Kle} and many others: it }{studies families of lines in $\mathbb P^n$ by means of the Pl\"ucker embedding of the Grassmannian $\mathbb G(1,n)$. In this language a \emph{line complex} is a hyperplane section of $\mathbb G(1,n)$; it is determined by a skew--symmetric bilinear form, whose kernel is the \emph{centre} of the complex, and, if $n$ is odd, the complexes with non--empty centre, called \emph{singular}, are parametrized by the dual variety of $\mathbb G(1,n)$, which is a Pfaffian hypersurface. The families of lines {in $\mathbb P^n$} of dimension $n-1$, called \emph{congruences}, come with their \emph{focal loci}, described by the points in which a line of the family meets the infinitely near ones. {The \emph{order} of a congruence is the number of lines of the congruence that pass through the general point of $\bbP^n$.} In this setting the classical Italian school found a wealth of beautiful varieties. To mention only a few examples related to this paper, the projected Veronese surface of $\bbP^4$ appears in G. Castelnuovo's study of the line geometry of $\bbP^4$ \cite{Cas} as the focal locus of a linear congruence of lines; F. Palatini \cite{Pal} computed the degrees of the loci of skew--symmetric matrices of given rank, and his name is attached to a famous threefold scroll of $\bbP^5$ (see \cite{M}); C. Segre \cite{Seg} studied the cubic threefold of $\bbP^4$ with ten nodes, its rational parametrization by the quadrics through five general points of $\bbP^3$, and the six families of lines on it, parametrized by quintic Del Pezzo surfaces, i.e., by linear sections of $\mathbb G(1,4)$; G. Fano \cite{Fa} proved that a general cubic threefold of $\bbP^4$ is birational to the Fano threefold of degree 14 which is a general linear section of $\mathbb G(1,5)$.}

{A landmark in this circle of ideas is the rationality theorem of F. Severi \cite{Sev} and L. Roth \cite{Ro}: the general linear section of codimension $c\leq n-k$ of the Grassmannian $\mathbb G(k,n)$ of $k$--dimensional linear subspaces of $\bbP^n$ is rational. Severi proved it for Grassmannians of lines, and Roth extended it to all Grassmannians with a simple and elegant argument: if $c=n-k$, through the general $(k-1)$--dimensional subspace of $\bbP^n$ there passes a unique $k$--dimensional subspace of the family parametrized by the section, and intersecting with a general hyperplane one gets a birational map onto $\mathbb G(k-1,n-1)$, which is a linear projection. In the same paper Roth described the inverse map and stated the extension of Fano's theorem to all Grassmannians $\mathbb G(1,n)$ with $n$ odd. Together with Semple's representations of Grassmannians by linear systems of quadrics \cite{sem}, these works are the main source of inspiration of the present paper.}

{We mainly consider the general linear sections $X=X_\Pi$ of $\mathbb G(1,n)$ with a linear space $\Pi$ of codimension $c$ equal to $n-1$, $n$ and $n+1$, which are, respectively, rational varieties, Fano varieties of index 1 and Calabi--Yau varieties. From the viewpoint of line geometry, $X_\Pi$ parametrizes the \emph{base lines} of the linear system $\mathcal L_\Pi$ of line complexes containing it, and a second object attached to $X_\Pi$ is the \emph{Palatini locus} of $\mathcal L_\Pi$, i.e., the locus of points lying in the centre of some complex of $\mathcal L_\Pi$, which, for $n$ odd, is swept out by the centres of the singular complexes, called \emph{Palatini lines}. The three cases fit together nicely (see Remark \ref{rem:unify}): for $c=n-1$ the base lines form a congruence of order one, whose focal locus is the Palatini locus; for $c=n$ the Palatini locus coincides with the locus described by the base lines, which is a hypersurface of degree $n-1$; for $c=n+1$ and $n$ odd the roles are exchanged, and the Palatini lines form a congruence of order one, whose focal locus is the locus described by the base lines. In all cases these loci are degeneracy loci of general morphisms $\mathcal O_{\bbP^n}^{\oplus c}\to \Omega_{\bbP^n}(2)$, which have been extensively studied (see \cite{BM, M, T1, T2} and \cite[\S 10.2]{Dol}), and this plays an important role in our arguments.}

{Part \ref{part:rat} is devoted to the rational sections {of $\mathbb G(k,n)$}. In Section \ref{ssec:k} we recall Roth's construction for the general linear section $X_\Pi$ of codimension $n-k$ of $\mathbb G(k,n)$, and the birational map $\rho_{\Pi,\pi}\colon X_\Pi\dasharrow \mathbb G(k-1,n-1)$ given by the intersection with a general hyperplane $\pi$. Roth showed that the inverse map is defined by a linear system of hypersurfaces of degree $n-k$ containing the codimension 2 subvariety $Z_{\Pi,\pi}$ of $\mathbb G(k-1,n-1)$ of the $(k-1)$--dimensional subspaces of $\pi$ contained in infinitely many subspaces of the family, and he took for granted that this linear system is complete. We describe $Z_{\Pi,\pi}$ as a degeneracy locus, we compute its class, and we prove that the linear system is indeed complete (see Theorem \ref{thm:conj}). In Section \ref{ssec:k1} we consider the case $k=1$, in which $Z_{\Pi,\pi}$ is the section with $\pi$ of the focal locus of the congruence of base lines, which coincides with the Palatini locus of $\mathcal L_\Pi$ (for $n=3,4,5$ it is a pair of skew lines, the projected Veronese surface and the Palatini scroll, respectively). As a by-product, we obtain a representation of $\mathbb G(1,n)$ on $\bbP^{2n-2}$ by forms of degree $n-1$ (see Remark \ref{rem:gen}), which, as far as we know, did not appear in the literature before.}

{Part \ref{part:fano} is devoted to the Fano case $c=n$ {for linear sections of $\mathbb G(1,n)$}. If $n=2k+1$ is odd, the singular complexes of $\mathcal L_\Pi\cong \bbP^{n-1}$ form a Pfaffian hypersurface $\mathcal C_\Pi$ of degree $k+1$, and we prove that $X_\Pi$ is birational to $\mathcal C_\Pi$ (see Theorem \ref{thm:bir}). For $k=2$ this is Fano's theorem, carefully revisited by V. A. Iskovskikh in \cite[Chapt. III, \S 1]{Isk} (see also \cite{B1, FS}), and in general it is the result stated by Roth in \cite[p. 107]{Ro}, of which we give a detailed proof. The key point is that the Palatini locus coincides with the locus $B_\Pi$ described by the base lines, which is a hypersurface of degree $n-1$ (see Proposition \ref{prop:pal}). In Sections \ref{sec:an} and \ref{sec:lines}, along the lines of Fano's and Iskovskikh's analysis, we study the indeterminacy loci and the linear systems defining the birational map $X_\Pi\dasharrow \mathcal C_\Pi$ and its inverse, and the lines on $X_\Pi$ and on $\mathcal C_\Pi$.}

{If $n=2k$ is even, the situation is similar and related to Segre's cubic threefold with ten nodes. In this case the general complex has a point as centre, and we prove in Section \ref{sec:even} that the Palatini locus, which again coincides with $B_\Pi$, is a rational hypersurface of degree $n-1$, explicitly parametrized by the $n+1$ Pfaffians of order $n$ of an antisymmetric matrix of order $n+1$ of linear forms on $\mathcal L_\Pi\cong \bbP^{n-1}$; for $n=4$ one recovers Segre's cubic and its parametrization by quadrics. As a consequence, $X_\Pi$ is unirational and birational to a general hypersurface of $\bbP^{n-1}$ defined by a Pfaffian of order $n$ of linear forms, which is rational by the results of Part \ref{part:CY} (for $n=4$ it is a quadric surface); hence $X_\Pi$ is actually rational (see Corollary \ref{cor:evenrat}).}

{Part \ref{part:CY} is devoted to the Calabi--Yau case $c=n+1$ {for linear sections of $\mathbb G(1,n)$}, with $n=2k+1\geq 5$. Now the singular complexes form a general Pfaffian hypersurface $\mathcal C_\Pi$ of degree $k+1$ of $\mathcal L_\Pi\cong \bbP^n$. We prove that a unique Palatini line passes through the general point of $\bbP^n$ (see Theorem \ref{thm:basic}), so that the general Pfaffian hypersurface of degree $k+1$ in $\bbP^{2k+1}$ is rational (see Theorem \ref{thm:rat}); for $k=2$ this is the rationality of Pfaffian cubic fourfolds proved in \cite{BD} (see also \cite{Tr}), and our argument is a direct generalization of theirs. More generally, the same argument, together with an elementary observation for $m\geq \frac {n(n+1)}2$, shows that the general Pfaffian hypersurface of degree $k+1$ in $\bbP^m$ is rational for every $m\geq n$ (this is the full statement of Theorem \ref{thm:rat}), although for $n+1\leq m\leq \frac {n(n+1)}2-1$ the corresponding linear sections of $\mathbb G(1,n)$ are of general type or empty (see Remark \ref{rem:ratm}). Moreover, combining Theorem \ref{thm:rat} with the results of Part \ref{part:fano}, we obtain that the general Fano linear section of index 1 of $\mathbb G(1,n)$ is rational for every even $n\geq 4$ (see Corollary \ref{cor:evenrat}). We then prove that the focal locus of the congruence of Palatini lines coincides with the locus $B_\Pi$ described by the base lines, which has dimension $n-2$ and degree ${n\choose 2}-1$ (see Theorems \ref{thm:prec} and \ref{thm:main}), that $X_\Pi$ is birational to a general hyperplane section of $B_\Pi$ (see Corollary \ref{cor:sec}), and that the hypersurfaces of degree $n-1$ containing such a hyperplane section give a birational parametrization of $\mathcal C_\Pi$ (see Theorem \ref{thm:inv}). Here $B_\Pi$ is the degeneracy locus of a general morphism $\mathcal O_{\bbP^n}^{\oplus (n+1)}\to \Omega_{\bbP^n}(2)$, and for such loci the questions on the Hilbert scheme posed in \cite[Question 2]{M} are left open there; although we do not address them here, this was one of our motivations. The case $n=5$, in which $X_\Pi$ is a K3 surface of degree 14, has been treated in \cite{BD}, from which we took inspiration (see also \cite{DP} and Remark \ref{rem:DP}).}

{The paper is organized as follows. Section \ref{sec:lc} contains preliminaries on line complexes and their linear systems, on foci of families of lines, on degeneracy loci of morphisms $\mathcal O_{\bbP^n}^{\oplus m}\to \Omega_{\bbP^n}(2)$ and on linear congruences of lines, which are used throughout the paper. Part \ref{part:rat} consists of Sections \ref{ssec:k} and \ref{ssec:k1}, Part \ref{part:fano} of Sections \ref{sec:fano}--\ref{sec:even}, and Part \ref{part:CY} of Sections \ref{sec:palat} and \ref{sec:base}. We work over the field of complex numbers, and our methods are mainly projective--geometric, in the spirit of the classical works quoted above, combined with some standard tools on degeneracy loci of morphisms of vector bundles.}

\section{Preliminaries}\label{sec:lc}

   In this section we recall some well known facts on Grassmannians and line complexes, and we prove a few general lemmata that will be used in all parts of the paper.

\subsection{Generalities} { We let $\mathbb G(k,n)\subset \bbP^{N(k,n)}$, with
 $$
 N(k,n)={{n+1}\choose {k+1}}-1,
 $$
be the Grassmann variety of $k$--dimensional linear subspaces in $\bbP^n=\bbP({\bf V})$,    where ${\bf V}$ is  a vector space of dimension $n+1$. The embedding  $\mathbb G(k,n)\subset \bbP^{N(k,n)}$ is the \emph{Pl\"ucker embedding} with  
$$\bbP^{N(k,n)}\cong \bbP(\wedge^{k+1}{\bf V}).$$

One has
$$
\dim (\mathbb G(k,n))=(k+1)(n-k),\,\,\, \deg (\mathbb G(k,n)) = \frac{((k+1)(n-k))!  \prod_{i=0}^{k} i!}{\prod_{i=0}^{k}(n-k+i)!}
$$
and the canonical bundle of $\mathbb G(k,n)$ is
$$
K_{\mathbb G(k,n)} \sim \mathcal O_{\mathbb G(k,n)}(-n-1).
$$}

\subsection{Line complexes} {We will often focus on the case $k=1$ of Grassmannians of lines.}  A hyperplane section of $\mathbb G(1,n)$ is called a \emph{line complex} of $\bbP^n$ and it is determined by a non--zero element of $\wedge^2\check {\bf V}$. It is well known that giving a line complex of $\bbP^n$ is equivalent to {giving} a linear rational map of projective spaces, say 
$$ \omega: \bbP^n \dasharrow {\check \bbP^n}, $$     such that,  for any point $p\in \bbP^n$ where $\omega$ is defined, the hyperplane $\omega(p)$ contains $p$.    Then the family of lines passing through $p$, and contained in $\omega(p)$, is the restriction of the line complex to the $(n-1)$-space $\mathbb P^{n-1}_p \subset \mathbb G(1,n)$, parametrizing all lines through $p$. By definition, a point $p \in \mathbb P^n$ is an indeterminacy point of $\omega$ {if and only if} the line complex contains $\mathbb P^{n-1}_p$. The locus of indeterminacy points of $\omega$ is the intersection
of all hyperplanes $\omega(p)$, with $p \in \mathbb P^n$. In the classical language of line geometry this is the  \emph{{centre}} of $\omega$. {If the line complex is determined by $h\in  \wedge^2\check {\bf V}$, then  the map  $\omega$ sends $p=[\bf p]\in \bbP^n$ (with ${\bf p}\in {\bf V}\setminus \{\bf 0\}$) to the hyperplane where the linear form $h(\bf p, \bf x)$ in $\bf x$ vanishes, and the {centre} of $\omega$ is the kernel of $h$.} \par 
Let us introduce homogeneous coordinates $[{\bf x}]=[x_0,\ldots, x_n]$ in $\bbP^n$ and dual coordinates $[{\bf u}]=[u_0,\ldots, u_n]$ in $\check \bbP^n$. We may interpret the coordinates as column vectors. Then a  line complex $\omega: \bbP^n\dasharrow {\check \bbP^n}$ is defined by an equation of the type
\begin{equation}\label{eq:compo}
{\bf u}= {\bf A}\cdot {\bf x}
\end{equation}
where ${\bf A}=(a_{ij})_{0\leq i,j\leq n}$ is an antisymmetric {non--zero} matrix. The $\frac {n(n+1)}2$ entries
$a_{ij}$, with $0\leq i<j\leq n$, can be interpreted as dual coordinates in the Pl\"ucker space  $\check \bbP^{\frac {n(n+1)}2-1}$. The {centre} of $\omega$ is the linear subspace of $\bbP^n$ with equation ${\bf A}\cdot {\bf x}={\bf 0}$. A line complex with equation \eqref {eq:compo} is said to be \emph{singular} if ${\bf A}$ has zero determinant, hence its {centre} is non--empty. 

\subsubsection{The odd case} Assume now that $n=2k+1$ is odd. To avoid trivial cases we will assume $k\geq 1$. 
In this case the general line complex of $\bbP^n$ is non--singular, hence it determines a bijective morphism
$$ \omega: \bbP^n\longrightarrow {\check \bbP^n}. $$ A general singular complex 
$\omega: \bbP^n \dasharrow {\check \bbP^n}$ has for {centre} a line $\ell_\omega$. This happens if and only if  the hyperplane $\pi_\omega$ corresponding to $\omega$ is tangent to $\mathbb G(1,n)$ at the point corresponding to $\ell_\omega$. Hence, the locus of singular complexes is the dual variety  ${\check {\mathbb G}(1,n)}$ of the Grassmannian $\mathbb G(1,n)$, and it is defined by the vanishing     of the \emph{Pfaffian} ${\rm Pf}({\bf A})$ of ${\bf A}$, which is      a hypersurface $\mathcal C_n$ of $\check \bbP^{\frac {n(n+1)}2-1}$ of degree $k+1$. 

Notice that the singular locus of $\mathcal C_n$ has codimension 5 in $\mathcal C_n$, and that     a point of $\mathcal C_n$ is smooth if and only if its {centre} is a line. The degree of the singular locus of $\mathcal C_n$ is \\
\begin{equation}\label{eq:degsing}
\deg({\rm Sing}(\mathcal C_n))=\frac{1} {240} {{n+1}\choose 3}{{n+2}\choose 2}(n+3)
\end{equation}
for $n\geq 5$ (see \cite[Prop. 12, (c)] {HT}).

\subsubsection{The even case} Assume next that $n=2k$ is even.  In this case any line complex of $\bbP^n$ is singular since  antisymmetric matrices of odd order never have  maximal rank. Actually a general line complex of $\bbP^n$  has for {centre} a single point of $\bbP^n$. 
A line complex with equation \eqref {eq:compo} is said to be \emph{very singular} if ${\rm rk}({\bf A})<n=2k$. A general such complex 
$\omega: \bbP^n \dasharrow {\check \bbP^n}$ has for {centre} a plane $\sigma_\omega$. As in the odd case, this happens if and only if  the hyperplane $\pi_\omega$ corresponding to $\omega$ is tangent to $\mathbb G(1,n)$. Hence, the locus of {very singular} complexes is the dual variety  ${\check {\mathbb G}(1,n)}$ of the Grassmannian $\mathbb G(1,n)$, and it is defined by the condition ${\rm rk}({\bf A})<n=2k$. This variety has codimension 3 in $\check \bbP^{\frac {n(n+1)}2-1}$ (i.e., the    Grassmannian     $\mathbb G(1,n)$ has \emph{dual defect} 2 if $n$ is even), and it turns out (see \cite {Pal} or \cite {HT}) that 

$$
\deg (\check {\mathbb G}(1,2k))=\frac {k(k+1)(2k+1)}6.
$$

\subsection{Linear sections and linear systems of line complexes}\label{sserc:ls}

Let us fix a linear subspace $\Pi$ of $\bbP^{\frac {n(n+1)}2-1}$, of codimension $c\geq 2$. 

The set $\mathcal L_\Pi$ of hyperplanes containing $\Pi$ cuts out on the Grassmannian a \emph{linear system of line complexes} of dimension 
$$
\dim (\mathcal L_\Pi)=c-1.$$

We will denote by  $X_\Pi$ the intersection of $\Pi$ with $\mathbb G(1,n)$, {   that is, } the set of lines that belong to all complexes of $\mathcal L_\Pi$. They are called the \emph{base lines} of $\mathcal L_\Pi$. Over $X_\Pi$ there is the universal family $\mathcal X_\Pi\longrightarrow X_\Pi$, and there is a natural projection morphism $\mathcal X_\Pi\longrightarrow \bbP^n$, whose image will be denoted by $B_{\Pi}$, and will be called the \emph{locus of base lines} of $\mathcal L_\Pi$. 

   Unless otherwise stated, we assume that     $\Pi$ is    a general linear space     of codimension $c$, with  $2n-3\geq c\geq 2$. In this     case the intersection of $\Pi$ with $\mathbb G(1,n)$ is proper,    meaning that      $\dim(X_\Pi)=2n-2-c$.    Furthermore,     $X_\Pi$  is smooth and irreducible and, by adjunction {(see \cite[(10.19)]{Dol})}, its canonical system is 
$$
K_{X_\Pi}   \sim \mathcal O_{X_\Pi}(c-n-1).
$$
Accordingly, if $X_\Pi$ is non--empty, the expected dimension of $B_{\Pi}$ is
$$
{\rm expdim}(B_{\Pi})=\min\{n,2n-1-c\}. 
$$

\subsubsection{The odd case}   Assume now that $n=2k+1$ is odd.    By the generality of $\Pi$, $\mathcal L_\Pi$ is not contained in ${\check {\mathbb G}(1,n)}$. Hence,  the scheme theoretic     intersection of 
$\mathcal L_\Pi$ with ${\check {\mathbb G}(1,n)}$ is a (Pfaffian) hypersurface $\mathcal C_\Pi$ of degree $k+1$ and dimension $c-2$ in $\mathcal L_\Pi$, which     we will call the \emph{singular locus} of $\mathcal L_\Pi$. If $c\geq 3$, $\mathcal C_\Pi$ is  irreducible. Since $\mathcal C_n$ is singular in codimension 5, $\mathcal C_\Pi$ is smooth if and only if $c\leq 6$. 

Assume $c\geq 3$. By the generality of $\Pi$, the general element $\omega\in \mathcal C_\Pi$ is a smooth point of $\mathcal C_n$,    and hence its {centre} is     a line $\ell_\omega$. This happens on the dense open subset $U_\Pi$ of  $\mathcal C_\Pi$ of smooth points of $\mathcal C_\Pi$, and there is a family $\phi_\Pi: \mathcal P_\Pi\longrightarrow U_\Pi$ such that for any $\omega\in U_\Pi$, the fibre {$\phi_{\Pi}^{-1}(\omega)$} is the line $\ell_\omega$. We will call this family the \emph{family of {centre}s} of $\mathcal L_\Pi$ or the \emph{Palatini family} of $\mathcal L_\Pi$ and a line in this family is called a \emph{Palatini line} of $\mathcal L_\Pi$.

\begin{remark} \label{rem:ram} Consider the projection $p_\Pi: \mathbb G(1,n) \dasharrow \Pi^\perp\cong \bbP^{c-1}$, {which is} defined by the linear system $\mathcal L_\Pi$, and therefore its indeterminacy locus is $X_\Pi$. Then it is clear that a line $\ell$ corresponding to the point $[\ell]\in \mathbb G(1,n)$ is the {centre} of a complex in the intersection of $\mathcal L_\Pi$ with ${\check {\mathbb G}(1,n)}$  if and only if it is in the ramification locus of $p_\Pi$. So the closure of $U_\Pi$ can be identified with the ramification locus $R_\Pi$ of $p_\Pi$.
\end{remark}

   We have the following:   

\begin{lemma}\label{lem:gen} If $X_\Pi$ is proper and smooth, then $R_\Pi\cap X_\Pi=\emptyset$. In particular the Palatini family and the    family of base lines     of $\mathcal L_\Pi$ have empty intersection. 
\end{lemma} 

\begin{proof} The assertion is trivial if $X_\Pi=\emptyset$. So we assume that this is not the case. Set $\mathbb G=\mathbb G(1,n)$ and notice  that the resolution of the indeterminacy of $p_\Pi$ is obtained by blowing up $\mathbb G$ along $X_\Pi$. Let
$\sigma: \tilde {\mathbb G} \to \mathbb G$
be this blow up, with  the exceptional divisor 
$E \subset \tilde {\mathbb G}$, {   that is, }  the $\mathbb P^{c-1}$-bundle associated to the normal bundle $N_{X_\Pi \vert \mathbb G} = \mathcal O_{X_\Pi}(1)^c$. Set $\tilde p_\Pi=p_\Pi \circ \sigma${, which is} now a morphism.

Let $D$ be the class of $\sigma^*(\mathcal O_{\mathbb G}(1))$.
Fix a basis $D_1, \ldots, D_c$ of the linear system $\vert D - E \vert$ and consider $H_i = \sigma_*(D_i)$, for $ i = 1, \ldots, c$. Then $X_\Pi$ is the smooth complete intersection $H_1 \cap \dots \cap H_c$ and
$H_1, \dots, H_c$ are hyperplane sections of $\mathbb G$ intersecting transversally at each $y \in 
X_\Pi$. Equivalently if $E_y = \sigma^{-1}(y)$, then $E_y$ is a projective space of dimension $c-1$ and the restriction map 
$$ H^0(\mathcal O_{\tilde {\mathbb G}}(D-E)) \longrightarrow H^0(\mathcal O_{E_y}(D - E)) $$
is an isomorphism. Since $\mathcal O_{E_y}(D-E) \cong \mathcal O_{\mathbb P^{c-1}}(1)$, the map $\tilde {p_\Pi}_{\vert E_y}: E_y \to \mathbb P^{c-1}$ is a linear isomorphism. Hence
${d\tilde p_{\Pi}}$ is surjective at each point of $ E$ and this implies the assertion. \end{proof}

\begin{remark} \label{rem:ram1} Note that $\dim(R_\Pi)=\dim(\mathcal C_\Pi)=c-2$ and, by Lemma \ref {lem:gen},  the restriction of the projection $p_\Pi: \mathbb G(1,n) \dasharrow  \bbP^{c-1}$ to $R_\Pi$ is finite. Hence the image of  $R_\Pi$ is a hypersurface in $\bbP^{c-1}$. It follows that a line $\ell$ corresponding to a general point $[\ell]\in R_\Pi$ is characterized by the property that the tangent space $T_{\mathbb G(1,n),\ell}$ to $\mathbb G(1,n)$ at the point $[\ell]$ and $\Pi$ span a hyperplane, hence they intersect in a linear space of dimension $2n-1-c$. \end{remark}

In particular, there  is  an obvious morphism $\mathcal P_\Pi\longrightarrow \bbP^n$ and the closure of the image of this morphism will be denoted by $P_\Pi$ and will be called the \emph{Palatini locus} of $\mathcal L_\Pi$. The expected dimension of $P_{\Pi}$ is
$$
{\rm expdim}(P_{\Pi})=\min\{n,c-1\}. 
$$

\subsubsection{The even case} Assume next that $n=2k$ is even and that $\Pi$ is general. So $\mathcal L_\Pi$ is not contained in ${\check {\mathbb G}(1,n)}$ and intersects ${\check {\mathbb G}(1,n)}$ in codimension 3. Then the general complex $\omega\in \mathcal L_\Pi$ has for {centre} a single point $p_\omega$. There is then an obvious rational map  $\omega\in \mathcal L_\Pi \dasharrow p_\omega\in \bbP^n$ and the closure of the image of this {rational map} will be denoted $P_\Pi$ and will be called the \emph{   Palatini locus   } of $\mathcal L_\Pi$. {As in the odd case, the expected dimension of $P_{\Pi}$ is $\min\{n,c-1\}$.}

   \subsection{Foci of families of lines}\label{ssec:foci}    We start by recalling a few facts about foci of families of lines. For the general theory of foci see \cite {CS}. 

Let  $\mathcal F$ be a family of dimension $n-1$ of lines in $\bbP^n$ filling up the whole space. If $\ell\cong \bbP^1$ is a general line of  $\mathcal F$, $n-1$ independent elements of the tangent space $T_{\mathcal F, \ell}$ give $n-1$ sections of the normal bundle $N_{\ell|\bbP^n}\cong \mathcal O_{\bbP^1}(1)^{n-1}$. The collection of these $n-1$ sections gives rise to a square matrix of order $n-1$ of linear forms on $\ell$, whose determinant is not identically 0. The zero locus of this determinant defines a 0--dimensional subscheme  of $\ell$ of length $n-1${, which is} called the \emph{scheme of foci} or the \emph{focal scheme} of the family $\mathcal F$ on $\ell$. {More generally, let $\mathcal I\longrightarrow \mathcal F$ be the universal family, whose fibre over $\ell\in\mathcal F$ is the line $\ell$, and let $f:\mathcal I\longrightarrow \bbP^n$ be the tautological morphism. For every line $\ell$ corresponding to a smooth point of $\mathcal F$, the same construction defines the subscheme of $\ell$ where the differential of $f$ is not injective: if $\ell$ is general, it is the focal scheme of $\mathcal F$ on $\ell$ considered in \cite{CS}, whereas it is the whole of $\ell$ if the above determinant vanishes identically on $\ell$. The \emph{focal locus} of the family $\mathcal F$ is the closure of the union of these subschemes, when $\ell$ varies among the smooth points of $\mathcal F$.}

{The following lemma, which will be applied several times, is a direct consequence of the definitions and of \cite{CS}.}

\begin{lemma}\label{lem:focinf} {Let $\mathcal F$ be a smooth, irreducible family of dimension $n-1$ of lines in $\bbP^n$, such that through the general point of $\bbP^n$ there passes a unique line of $\mathcal F$. Then the focal locus of $\mathcal F$ coincides with the closure of the set of points $p\in \bbP^n$ such that there are infinitely many lines of $\mathcal F$ passing through $p$. In particular, it has dimension at most $n-2$.}
\end{lemma}

\begin{proof} {Let $\mathcal I\longrightarrow \mathcal F$ be the universal family, whose fibre over $\ell\in \mathcal F$ is the line $\ell$, and let $f:\mathcal I\longrightarrow  \bbP^n$ be the tautological morphism, which is birational by assumption.} From  \cite {CS} it follows that, for a general line $\ell$ {of $\mathcal F$}, that we can see as a fibre of {$\mathcal I\longrightarrow  \mathcal F$}, the focal {scheme of $\mathcal F$ on $\ell$} coincides with the subscheme of $\ell$ where the differential of {$f$} is not of maximal rank $n$. {By the definition of the focal locus given above, the same holds for any line of $\mathcal F$, so that the focal locus is the closure of the image via $f$ of the set of points of $\mathcal I$ where the differential of $f$ is not of maximal rank.} Since {$f$} is birational, its differential drops rank exactly at those points belonging to positive dimensional fibres {(by Zariski's Main Theorem, since $\bbP^n$ is normal)}. 
\end{proof}

\subsection[Linear systems of line complexes and degeneracy loci]{{Linear systems of line complexes and degeneracy loci}}\label{ssec:omega}

{
Let $\mathcal L$ be a linear system of line complexes of $\bbP^n=\bbP({\bf V})$ of dimension $m-1$. Consider the twisted Euler sequence 
$$
0\longrightarrow \mathcal O_{\bbP^n}(-1)\longrightarrow \mathcal  O_{\bbP^n}^{\oplus n+1}\longrightarrow T_{\bbP^n}(-1)\longrightarrow 0.
$$
The projectivization of the fibre of $T_{\bbP^n}(-1)$ at a point $p\in \bbP^n$  parametrizes the set of lines in $\bbP^n$ through the point $p$. Consider now a hyperplane $H\in \mathcal L$. We claim that it defines a  map of vector bundles 
$$
\alpha_H: T_{\bbP^n}(-1) \otimes \mathcal O_{\bbP^n}(-1)\cong T_{\bbP^n}(-2) \to \mathcal O_{\bbP^n}.
$$
Indeed, we may identify $H$ with a bilinear form  $h\in \wedge^2 {\check {\bf V}}$. By the above considerations, a point in $T_{\bbP^n}(-1) \otimes \mathcal O_{\bbP^n}(-1)\cong T_{\bbP^n}(-2)$ can be identified with ${\bf p}\wedge {\bf v}$, with $[{\bf p}]$ defining a point in $\bbP^n$ and ${\bf p}\wedge {\bf v}$ a line through $[{\bf p}]$. Then $\alpha_H({\bf p}\wedge {\bf v}):=h({\bf p},{\bf v})$. 

If $H_1,\ldots, H_m$ are linearly independent hyperplanes spanning $\mathcal L$, corresponding to $h_1,\ldots, h_m\in  \wedge^2 {\check {\bf V}}$, the datum of $\alpha_{H_1}, \ldots, \alpha_{H_m}$ gives a map 
\begin{equation}\label{eq:alphaL}
\alpha_{\mathcal L}: T_{\bbP^n}(-2) \longrightarrow \mathcal O_{\bbP^n}^{m}
\end{equation}
and dually a map
\begin{equation}\label{eq:wL}
w_{\mathcal L}: \mathcal O_{\bbP^n}^{m}\longrightarrow \Omega_{\bbP^n}(2).
\end{equation}
Recall that $H^0(\bbP^n,\Omega_{\bbP^n}(2))$ is naturally identified with $ \wedge^2 {\check {\bf V}}$, i.e., with the space of antisymmetric matrices of order $n+1$, and $w_{\mathcal L}$ is the morphism defined by $h_1,\ldots, h_m$ (see \cite[\S 2]{BM}, \cite[\S 1.2]{M}, \cite[\S 10.2.3]{Dol}). If $p={[{\bf p}]}$, the kernel of the linear map induced by $w_{\mathcal L}$ on the fibres at $p$ consists of the $m$--tuples $(\lambda_1,\ldots,\lambda_m)$ such that ${\bf p}$ is in the kernel of $\sum_i\lambda_ih_i$, i.e., such that $p$ lies in the centre of the complex $\sum_i\lambda_iH_i$. Hence:\\
\begin{inparaenum}
\item [$\bullet$] if $m\leq n$, the degeneracy locus of $w_{\mathcal L}$, i.e., the locus where $w_{\mathcal L}$ is not injective, is the locus of points of $\bbP^n$ lying in the centre of some complex of $\mathcal L$ (the \emph{center variety} of $\mathcal L$, see \cite[Prop. 10.2.17]{Dol}); in particular, if $\Pi$ is general of codimension $c\leq n$, both for $n$ odd and even, the Palatini locus $P_\Pi$ is the degeneracy locus of the morphism $w_{\mathcal L_\Pi}:\mathcal O_{\bbP^n}^{\oplus c}\to \Omega_{\bbP^n}(2)$ defined by $c$ complexes spanning $\mathcal L_\Pi$;\\
\item [$\bullet$] if $m=n+1$ and $\mathcal L$ is general, the morphism $w_{\mathcal L}$ is generically surjective, i.e., $\alpha_{\mathcal L}$ is injective, and its degeneracy locus, i.e., the locus where $w_{\mathcal L}$ is not surjective, is, if not empty, of codimension at most 2; it will appear in Part \ref{part:CY} as the focal locus of the Palatini family (see Theorem \ref{thm:prec}).
\end{inparaenum}

The following well known fact will be used several times.
}

\begin{lemma}\label{lem:HB} {Let $Y$ be a smooth variety and let $\varphi: A\to B$ be a morphism of vector bundles on $Y$ of ranks $r$ and $r+1$ respectively. Assume that the degeneracy locus $D(\varphi)$ of $\varphi$, i.e., the subscheme of $Y$ defined by the maximal minors of $\varphi$, is either empty or of codimension 2 in $Y$. Then there is an exact sequence
$$
0\longrightarrow A\stackrel{\varphi}\longrightarrow B\stackrel{\mu}\longrightarrow \mathcal I_{D(\varphi)|Y}\otimes \det(B)\otimes \det(A)^\vee\longrightarrow 0,
$$
where, if $a_1,\ldots, a_r$ is a local frame of $A$, then {$$ \mu(x)=x\wedge\varphi(a_1)\wedge\cdots\wedge\varphi(a_r)\otimes(a_1\wedge\cdots\wedge a_r)^\vee.$$}}
\end{lemma}

\begin{proof} {The map $\mu: B\to \det(B)\otimes\det(A)^\vee$ does not depend on the choice of the local frame; in local frames its components are, up to sign, the maximal minors of $\varphi$, so that its image is $\mathcal I_{D(\varphi)|Y}\otimes \det(B)\otimes \det(A)^\vee$, and $\mu\circ\varphi=0$. The exactness of the resulting complex is a local question. Since $Y$ is smooth, hence Cohen--Macaulay, and $D(\varphi)$ is either empty or of codimension $2$, the ideal generated by the maximal minors of $\varphi$ has depth at least $2$ at each point of $Y$, and the claim  follows from the Hilbert--Burch theorem (see \cite[Thm.~20.15]{Eis}). Note that the Chern classes of $A$ and $B$ only determine the twist $\det(B)\otimes\det(A)^\vee$, whereas the exactness depends on the codimension of $D(\varphi)$.}
\end{proof}

\begin{remark}\label{rem:lit} {Degeneracy loci of general morphisms $\mathcal O_{\bbP^n}^{\oplus m}\to \Omega_{\bbP^n}(2)$ have been extensively studied: see \cite[\S 2]{BM} and \cite[\S 1.2]{M} (and the references therein) for $2\leq m\leq n$, and \cite[\S 2.3]{T1}, \cite[\S 3.2]{T2} for $m\leq n-1$, in which case the degeneracy locus has codimension at least 2. In this paper we will meet the case $m=n-1$ (see \S \ref{ssec:cong} and Section \ref{ssec:k1}), in which the degeneracy locus is the focal locus of a congruence of lines of order one, a case related also to \cite{DP, DM}; the case $m=n$ (see Part \ref{part:fano}), in which the degeneracy locus is the Palatini hypersurface; and the case $m=n+1$ (see Part \ref{part:CY}), which is not covered by the aforementioned references. See also \cite[\S 10.2.3]{Dol}, where, for $m\leq n$, the degeneracy locus is described as the center variety of the corresponding linear system of complexes and its degree is computed: by \cite[(10.22)]{Dol}, it is equal to $\sum_{i=0}^{n-m+1}(-1)^i{{n-i}\choose {m-1}}$. We will point out the overlaps with the literature when they occur.}
\end{remark}

\subsection{Some useful lemmata}

The following lemmata will be useful later. 
   The first {four} collect some elementary arguments that will be used several times in the paper.

\begin{lemma}\label{lem:hyp} Let $\mathcal L$ be a linear system of line complexes of $\bbP^n$, let $p\in \bbP^n$ and let $\mathcal L_p\subseteq \mathcal L$ be the (possibly empty) linear subsystem of the complexes whose {centre} contains $p$. Then the hyperplanes $\omega(p)$, with $\omega\in \mathcal L\setminus \mathcal L_p$, form a linear system $\mathcal H_{\mathcal L,p}$ of hyperplanes containing $p$, with
$$
\dim(\mathcal H_{\mathcal L,p})=\dim(\mathcal L)-\dim(\mathcal L_p)-1
$$
(with the convention $\dim(\emptyset)=-1$), and the base lines of $\mathcal L$ passing through $p$ are the lines through $p$ contained in the base locus of $\mathcal H_{\mathcal L,p}$. In particular, if $\dim(\mathcal H_{\mathcal L,p})=n-2$, there is a unique base line of $\mathcal L$ passing through $p$.
\end{lemma}

\begin{proof} If $p=[{\bf p}]$, the map $\omega\mapsto \omega(p)$ is the projectivization of the linear map ${\bf A}\mapsto {\bf A}\cdot {\bf p}$ (see \eqref{eq:compo}), defined on the vector space of antisymmetric matrices corresponding to $\mathcal L$, whose kernel corresponds to $\mathcal L_p$. This proves the assertion on the dimension of $\mathcal H_{\mathcal L,p}$, whose elements contain $p$ because ${}^t{\bf p}\cdot {\bf A}\cdot {\bf p}=0$. A line through $p$ belongs to a complex $\omega\in \mathcal L\setminus \mathcal L_p$ if and only if it is contained in $\omega(p)$; moreover any line through $p$ belongs     to all complexes in $\mathcal L_p$. This proves the    assertion regarding     the base lines through $p$. Finally, if $\dim(\mathcal H_{\mathcal L,p})=n-2$, the base locus of $\mathcal H_{\mathcal L,p}$ is a line through $p$, and the last statement follows. \end{proof}

\begin{lemma}\label{lem:line} Let $r$ be a line of $\bbP^n$ and let $\mathcal L'$ be a general linear system of dimension $n-2$ of line complexes of $\bbP^n$ containing $r$. Then the points of $r$ lying in the {centre} of some complex of $\mathcal L'$ are $n-1$, counted with multiplicity. 
\end{lemma}

\begin{proof} By choosing coordinates we may assume that $r$ has equations $x_2=\cdots=x_n=0$. The elements of $\mathcal L'$ are of the form \eqref {eq:compo}, with $a_{01}=0$ because the complexes of $\mathcal L'$ contain $r$, and with the other entries of the matrix $\bf A$ that are linear forms on $\mathcal L'\cong \bbP^{n-2}$. Setting $x_2=\cdots=x_n=0$ in \eqref {eq:compo}, we see that a point $[x_0,x_1,0,\ldots,0]$ of $r$ lies in the {centre} of $\omega\in \mathcal L'$ if and only if 
$$
\begin{aligned}
0&=a_{20}x_0+a_{21}x_1&\cr
&\qquad\quad\,\,\,{\ldots}& \cr
0&=a_{n0}x_0+a_{n1}x_1.&\cr
\end{aligned}
$$
This system has a non--trivial solution if and only if the matrix
$$
\left ( 
\begin{matrix}
a_{20}&a_{21}\\
\ldots&\ldots\\
a_{n0}&a_{n1}\\
\end{matrix}
\right)
$$
has rank smaller than 2. By the generality of $\mathcal L'$, the entries of this matrix are general linear forms on $\mathcal L'\cong \bbP^{n-2}$, hence this happens on a finite locus of length $n-1$ in $\mathcal L'$, at whose points the matrix has rank exactly 1, so that each of these complexes has exactly one point of $r$ in its {centre}. Two distinct complexes of $\mathcal L'$ cannot have the same point of $r$ in their {centre}s, otherwise this would happen for all complexes of the pencil they span, and the above locus would not be finite. This concludes the proof. \end{proof}

\begin{lemma}\label{lem:proj} Let $\mathcal F$ be an irreducible variety and let $t\in \mathcal F\dasharrow \ell_t\in \mathbb G(1,n)$ be a rational map. Let $W\subseteq \bbP^n$ be the closure of the union of the lines $\ell_t$, and assume that through the general point of $W$ there passes $\ell_t$ for a unique $t\in \mathcal F$. Then, for a general hyperplane $\Sigma$ of $\bbP^n$, the rational map $\mathcal F\dasharrow W\cap \Sigma$ sending $t$ to $\ell_t\cap \Sigma$ is birational.
\end{lemma}

\begin{proof} Let $\mathcal I\subseteq \mathcal F\times \bbP^n$ be the closure of the set of pairs $(t,x)$ with $x\in \ell_t$, and let $f:\mathcal I\longrightarrow W$ be the projection, {which is} birational by assumption. Let $\mathcal I_\Sigma=f^{-1}(\Sigma)$. Since $\Sigma$ is general, the general line $\ell_t$ is not contained in $\Sigma$, hence the projection $\mathcal I_\Sigma\longrightarrow \mathcal F$ is birational, with inverse $t\mapsto (t, \ell_t\cap \Sigma)$. On the other hand $f$ is an isomorphism over a dense open subset $W^\circ$ of $W$ and, $\Sigma$ being general, $W^\circ\cap \Sigma$ is dense in $W\cap \Sigma$, so that $f$ induces a birational map $\mathcal I_\Sigma\longrightarrow W\cap \Sigma$.  \end{proof}

\begin{lemma}\label{lem:projR} {Let $\Pi$ be a general linear subspace of $\bbP^{\frac {n(n+1)}2-1}$ of codimension $c\leq 2n-4$, let $\pi$ be a general hyperplane of $\bbP^n$, identify $\mathbb G(1,n-1)$ with the Grassmannian of lines contained in $\pi$ and set 
$$
\mathcal R_{\Pi,\pi}:=X_\Pi\cap \mathbb G(1,n-1)=\Pi\cap \mathbb G(1,n-1).
$$ 
Then:\\
\begin{inparaenum}[(i)]
\item $\mathcal R_{\Pi,\pi}$ has dimension $2n-4-c$ and degree $\deg(\mathbb G(1,n-1))=\frac {(2n-4)!}{(n-1)!(n-2)!}$, and it spans a linear space of dimension $\frac{n(n-1)}2-1-c$;\\
\item $\mathcal R_{\Pi,\pi}$ is the indeterminacy locus of the rational map
$$
\rho_{\Pi,\pi}: X_\Pi\dasharrow \pi, \quad [\ell]\mapsto \ell\cap \pi,
$$
and $\rho_{\Pi,\pi}$ is the restriction to $X_\Pi$ of the linear projection of $\Pi$ from $\langle \mathcal R_{\Pi,\pi}\rangle$ to $\pi\cong \bbP^{n-1}$.
\end{inparaenum}}
\end{lemma}

\begin{proof} {By the generality of $\Pi$, $\Pi\cap \langle \mathbb G(1,n-1)\rangle$ is a general linear subspace of codimension $c$ of the Pl\"ucker space $\langle \mathbb G(1,n-1)\rangle\cong \bbP^{\frac{n(n-1)}2-1}$, and (i) follows, because the non--empty section of an irreducible, non--degenerate variety with a general linear space spans this linear space. The first assertion of (ii) is immediate.} We claim that the map  $\rho_{\Pi,\pi}$ is linear in the Pl\"ucker coordinates of $X_\Pi$. Indeed, we may assume that the hyperplane $\pi$ has equation $x_0=0$. If $\ell$ is a general line of $X_\Pi$, we may assume that it has Pl\"ucker coordinates given by the order two determinants of a matrix of the form
$$
\left ( 
\begin{matrix}
1&0&a_2&a_3&\ldots &a_n\\
1&b_1&0&b_3&\ldots &b_n\\
\end{matrix}
\right)
$$
so that these  Pl\"ucker coordinates are 
$$
[b_1,-a_2, b_3-a_3, \ldots, b_n-a_n, \ldots].
$$
The coordinates of the intersection point of the line $\ell$ with $\pi$ are
$$
[b_1, -a_2, b_3-a_3,\ldots, b_n-a_n] 
$$
and this proves our claim. {More precisely, the coordinates of $\ell\cap \pi$ are the Pl\"ucker coordinates $p_{01},\ldots, p_{0n}$ of $\ell$, which vanish simultaneously exactly on the Pl\"ucker space of $\mathbb G(1,n-1)$. Hence $\rho_{\Pi,\pi}$ is the restriction to $X_\Pi$ of the projection of $\Pi$ from $\Pi\cap\langle\mathbb G(1,n-1)\rangle=\langle \mathcal R_{\Pi,\pi}\rangle$, which proves (ii).} \end{proof}

As a sanity check we see that 
the assertion of {Lemma \ref {lem:projR}} fits with the dimensions of the ambient spaces. Indeed
$$
\dim(\Pi)-\dim (\langle {\mathcal R_{\Pi,\pi}}\rangle)-1={\Big(\frac{n(n+1)}2-1-c\Big)-\Big(\frac{n(n-1)}2-1-c\Big)-1}=n-1.
$$

The following lemmata will be used in Part \ref{part:fano}.

\begin{lemma}\label{lem:1} Let $n=2k+1$. Let $V_\Sigma$ be the  variety of singular complexes of $\bbP^n$ with {centre} a line contained in a fixed hyperplane $\Sigma\cong \bbP^{n-1}$. Then $V_\Sigma$ is a cone with vertex a linear space of dimension $n-1$ over the dual of $\mathbb G(1,n-1)$. Hence $V_\Sigma$  is an irreducible variety with
$$
\dim(V_\Sigma)= \frac {n(n+1)}2-4,\,\,\, \deg (V_\Sigma)= \frac {k(k+1)(2k+1)}6.
$$
\end{lemma} 

\begin{proof} Let $\omega$ be a singular complex with {centre} a line $\ell_\omega$ contained in $\Sigma$. Then the restriction $\omega_\Sigma$ of $\omega$ to $\Sigma$ is singular with {centre} at least a plane $\sigma$ containing $\ell_\omega$. Conversely, if a complex $\omega$ of $\bbP^n$ restricts  along  $\Sigma$ to a complex $\omega_\Sigma$  with {centre} a plane $\sigma$, then $\omega$ is singular  and its {centre} contains a line which is   in $\Sigma$. Indeed, suppose that $\Sigma$ has equation $x_0=0$. Then if $\omega$, with equation
\eqref {eq:compo},  restricts along  $\Sigma$ to a complex $\omega_\Sigma$  with {centre} a plane $\sigma$, then this implies that the antisymmetric matrix ${\bf A}'=(a_{ij})_{1\leq i,j\leq n}$ has rank $n-3$. This implies that ${\bf A}=(a_{ij})_{0\leq i,j\leq n}$ is singular, hence $\omega$ has {centre} a line and it is immediate that this line sits in $\Sigma$. In conclusion,  the {centre} of a singular line complex $\omega$  contains a line contained in $\Sigma$ if and only if $\omega$ restricts to $\Sigma$ to a singular complex (with {centre} at least a plane). 

Next identify $\mathbb G(1,n-1)\subset \mathbb G(1,n)$ with the set of lines contained in $\Sigma$, then $V_\Sigma$ is the dual variety of $\mathbb G(1,n-1)$ considered in $\bbP^{\frac {n(n+1)}2-1}$, hence $V_\Sigma\subset {\check \bbP^{\frac {n(n+1)}2-1}}$ is nothing but the cone over $\check {\mathbb G}(1,n-1)\subset {\check \bbP^{\frac {n(n-1)}2-1}}$ with vertex the linear space of dimension $n-1$ orthogonal in ${\check \bbP^{\frac {n(n+1)}2-1}}$ to the span $\bbP^{\frac {n(n-1)}2-1}$ of  $\mathbb G(1,n-1)$. Hence $V_\Sigma$  has codimension 3 in $\check \bbP^{\frac {n(n+1)}2-1}$ and degree equal to the degree of $\check {\mathbb G}(1,n-1)$. 
\end{proof}

\begin{lemma}\label{lem:1.2} Let $n=2k+1$. Let $P\cong \bbP^{n-2}$ be a subspace of $\bbP^n$. Then the variety $\Lambda_P$ of complexes of $\bbP^n$ that restrict to a singular complex of $P$ is an irreducible hypersurface of ${\check \bbP^{\frac {n(n+1)}2-1}}$ of degree $k$.
\end{lemma}

\begin{proof}  It is immediate that $\Lambda_P$ is the cone over the dual of the Grassmannian of lines of $P$ with vertex {the linear space of dimension $2n-2$ of the complexes containing all lines of $P$}.
\end{proof}

\begin{lemma}\label{lem:2} Let $n=2k+1$. Let $W_P$ be the  variety of singular complexes of $\bbP^n$ with {centre} intersecting a fixed subspace  $P\cong \bbP^{n-2}$ in $\bbP^n$. Then  
 $W_P$  is an irreducible variety with
$$
\dim(W_P)= \frac {n(n+1)}2-3,\,\,\, \deg (W_P)= k(k+1).
$$
More precisely, $W_P$ is a complete intersection of type $(k,k+1)$ in ${\check \bbP^{\frac {n(n+1)}2-1}}$.\end{lemma}

\begin{proof} Let us consider the intersection of $\Lambda_P$ (see Lemma \ref {lem:1.2}) with ${\check {\mathbb G}(1,n)}$. We claim that this intersection coincides with $W_P$. Clearly $W_P$ is contained in this intersection. Let us prove the converse inclusion. Take a point $\omega$ in this intersection. Then $\omega$ has a line $\ell_\omega$ in $\bbP^n$ contained in its {centre} and its restriction $\omega_P$ to $P$ has a  line $\ell_{\omega,P}$ in its {centre}. If $\ell_\omega$ intersects $P$, we are done. Otherwise, we notice that any  point $p$ of $\ell_{\omega,P}$ is an indeterminacy point for $\omega$. Indeed, all lines in $P$ passing through $p$ are in the complex $\omega$. Also,  all lines through $p$ and through any point of $\ell_\omega$ are in $\omega$. This yields that $p$ is in the {centre} of $\omega$. Since $\ell_\omega$  is also in the {centre} of $\omega$, then the {centre} of $\omega$ contains $\langle \ell_\omega, \ell_{\omega,P}\rangle$, and this proves the assertion. \end{proof}

Finally we record the following:

\begin{lemma}\label{lem:0} Consider the Schubert variety $\Omega_{(2,n)}$ in $\mathbb G(1,n)$, i.e., $\Omega_{(2,n)}$ is the set of all lines in $\bbP^n$ that intersect a given plane. Then
$$
\deg(\Omega_{(2,n)})=\frac {(n+1)(n-2)}2.
$$
\end{lemma}

\begin{proof} This follows from a standard iterated application of Pieri's formula and can be left to the reader. \end{proof}

\subsection[Linear congruences of base lines]{{Linear congruences of base lines}}\label{ssec:cong}

{In this subsection we consider a general linear system $\mathcal L$ of line complexes of $\bbP^n$ of dimension $n-2$, i.e., $\mathcal L=\mathcal L_\Pi$ with $\Pi$ a general linear subspace of codimension $n-1$, and the family of its base lines, which has dimension $n-1$. We denote by $V_{\mathcal L}$ the Palatini locus of $\mathcal L$, i.e., the set of points of $\bbP^n$ lying in the centre of some complex of $\mathcal L$, which is the degeneracy locus of the morphism $w_{\mathcal L}:\mathcal O_{\bbP^n}^{\oplus n-1}\to \Omega_{\bbP^n}(2)$ (see \S \ref{ssec:omega}). If $n$ is odd, $V_{\mathcal L}$ is described by the centre lines of the singular complexes in $\mathcal L$, which form a hypersurface of $\mathcal L$; if $n$ is even, it is described by the centres of the complexes of $\mathcal L$, which are points for the general complex. In both cases an obvious dimension count shows that $\dim(V_{\mathcal L})\leq n-2$.}

\begin{lemma}\label{lem:3} {Let $\mathcal L$ be a general  linear system of complexes of $\bbP^n$ of dimension $n-2$. Then $V_\mathcal L$ has dimension $n-2$ and}
$$
\deg(V_\mathcal L)=\frac {n(n-3)}2+2.
$$
\end{lemma}

\begin{proof} We suppose that the complexes of $\mathcal L$ have equation \eqref {eq:compo}, depending linearly on ${n-1}$ {homogeneous} parameters. Fix a general plane $\sigma\subset \bbP^n$. We have to compute the number of intersection points of $\sigma$ with $V_\mathcal L$ or, in other words, the number of points on $\sigma$ that are indeterminate for a complex of $\mathcal L$. By changing coordinates, we may assume that $\sigma$ has equations $x_3=\cdots=x_n=0$. Then we have to compute the number of non--trivial solutions of the system 
$$
\begin{aligned}
0&= a_{01}x_1+a_{02}x_2&\cr
0&= a_{10}x_0+a_{12}x_2&\cr
0&=a_{20}x_0+a_{21}x_1&\cr
0&=a_{30}x_0+a_{31}x_1+a_{32}x_2.&\cr
&\qquad\quad\,\,\,{\ldots}& \cr
0&=a_{n0}x_0+a_{n1}x_1+a_{n2}x_2.&\cr
\end{aligned}
$$
The first three equations have the solution 
$$
x_0=a_{12},\,\, x_1=-a_{02}, \,\, x_2=a_{01}.
$$
Plugging into the other equations we have the system
$$
a_{i0}a_{12}{+}a_{i1}a_{20}+a_{i2}a_{01}=0,\,\, 3\leq i\leq n.
$$
So we have to intersect $n-2$ quadrics in a $\bbP^{n-2}$ off the linear subspace of codimension 3 with equations 
$$
a_{12}=a_{02}=a_{01}=0
$$
whose points do not correspond to valid solutions  because they do not give any point in the plane $\sigma$ with equations  $x_3=\cdots=x_n=0$. This number is 
$$
\frac {n(n-3)}2+2
$$
by \cite [Prop. 13.5]{EH}. {In particular, since $\dim(V_{\mathcal L})\leq n-2$ and $V_{\mathcal L}$ meets a general plane in a positive number of points, $V_{\mathcal L}$ has dimension $n-2$.} \end{proof}

\begin{lemma}\label{   centre} {Let $\mathcal L$ be a general 
linear system of complexes of $\bbP^n$ of  dimension $n-2$.} Then:\\
\begin{inparaenum}[(i)]
\item the family of base lines of $\mathcal L$ has dimension $n-1$;\\
\item the locus of base lines of $\mathcal L$ fills up the whole $\bbP^n$;\\
\item through the general point of $\bbP^n$ passes a unique base line of $\mathcal L$;\\
\item on a general base line of $\mathcal L$ there are $n-1$ foci of the family of base lines of $\mathcal L$ and the focal locus of the family is the variety $V_\mathcal L$ of dimension $n-2$ {introduced above} (see Lemma \ref {lem:3}), whose intersection points with the general base line are therefore $n-1${;\\
\item through the general point $p$ of $V_\mathcal L$ there passes a one--dimensional family of base lines of $\mathcal L$, namely the pencil of lines through $p$ in {a suitable } plane}. 
\end{inparaenum}
\end{lemma}

\begin{proof} Assertion (i) is obvious. {Since $\dim(V_{\mathcal L})\leq n-2$, a general point $p\in\bbP^n$ is not in the {centre} of any complex of $\mathcal L$, hence, by Lemma \ref{lem:hyp} (with $\mathcal L_p=\emptyset$), there is a unique base line of $\mathcal L$ passing through $p$. This proves (ii) and (iii).} Next, a family of dimension $n-1$ of lines in $\bbP^n$ filling up $\bbP^n$, has a focal scheme of length $n-1$ on the general line. 
To prove that the focal locus coincides with $V_\mathcal L$ it suffices to show that a general base line intersects $V_\mathcal L$ in $n-1$ points{, because, by Lemma \ref{lem:hyp} (with $\mathcal L_p\neq\emptyset$), through each point $p\in V_{\mathcal L}$ there pass infinitely many base lines of $\mathcal L$, so that $p$ is a focal point by Lemma \ref{lem:focinf}}. Let $r$ be a general base line of $\mathcal L$. By the generality of $\mathcal L$ and of $r$, the linear system $\mathcal L$ is general among the linear systems of dimension $n-2$ of complexes containing $r$, and the claim follows from Lemma \ref{lem:line}. 

{As for (v), consider the closure $\Gamma\subset \mathcal L\times\bbP^n$ of the set of pairs $(\omega,p)$ such that the centre of $\omega$ is a line, if $n$ is odd, or a point, if $n$ is even, and $p$ lies in the centre of $\omega$. Then $\Gamma$ has dimension $n-2$ and it maps onto $V_\mathcal L$, which has dimension $n-2$ by Lemma \ref{lem:3}. Hence a general point $p\in V_{\mathcal L}$ lies in the centre of finitely many complexes of $\mathcal L$, hence of exactly one complex, because the complexes of $\mathcal L$ whose centre contains $p$ form a linear system $\mathcal L_p$. By Lemma \ref{lem:hyp}, the linear system $\mathcal H_{\mathcal L,p}$ has dimension $n-3$, its base locus is a plane through $p$, and the base lines of $\mathcal L$ through $p$ are the lines through $p$ in this plane.}
\end{proof}

\begin{remark}\label{rem:cong} We observe that, {in the hypotheses of Lemma \ref {   centre}}, the base lines of $\mathcal L$ form a  congruence of lines which has order one, has $V_\mathcal L$ as focal locus and consists of the $(n-1)$--secant lines of $V_\mathcal L$. In this framework, {the degree of $V_\mathcal L$ (see Lemma \ref{lem:3}) has already been computed in} \cite[\S 2]{BM} and \cite[Prop. 2.4]{DP} (see also \cite[Prop. 5]{DM}), {and} parts of Lemma \ref {   centre} are already present in \cite[Cor. 2.3, Prop. 2.4]{DP}, \cite[Prop. 6]{DM} and \cite[\S 1.2]{M}. We have included {elementary proofs} for {the reader's convenience. The degree of $V_\mathcal L$ can also be computed with Porteous' formula, see Proposition \ref{prop:class} (with $k=1$) and Theorem \ref{thm:precis}, and it is the case $m=n-1$ of \cite[(10.22)]{Dol}. For $n=4$, the fact that the base lines of $\mathcal L$ are the trisecant lines of the projected Veronese surface $V_{\mathcal L}$ is \cite[Thm. 2.1.6]{Dol} (see also \cite[Ex. 10.2.19]{Dol}).}
\end{remark}

The following lemma was originally proved in \cite[\S 2]{Ro}. 

\begin{lemma}\label{lem:roth} Let $r$ be a general line in $\bbP^n$. The {base lines of $\mathcal L$} intersecting $r$ describe a rational normal scroll of degree $n-1$.
\end{lemma}

\begin{proof} The degree $m$ {of this scroll} equals  the number of lines that intersect the line $r$ and also $n$ given general subspaces of dimension $n-2$. So 
$$
m=\sigma_1^n\cdot \sigma_{n-2}
$$
where $\sigma_1=\sigma_{1,0}$ is the hyperplane section of $\mathbb G(1,n)$ and $\sigma_{n-2}=\sigma_{n-2,0}$ is the Schubert cycle of lines meeting a line. The assertion follows easily by an iterated application of Pieri's formula. 
\end{proof}

\begin{remark}\label{rem:unify} {Let $\Pi$ be a general linear subspace of codimension $c$ of $\bbP^{\frac {n(n+1)}2-1}$. The cases $c=n-1$, $c=n$ and $c=n+1$, considered in the three parts of this paper, show a nice symmetry between the locus $B_\Pi$ of base lines of $\mathcal L_\Pi$ and the Palatini locus $P_\Pi$. If $c=n-1$, the base lines fill up $\bbP^n$ and form a congruence of order one, whose focal locus is the Palatini locus $P_\Pi$, of dimension $n-2$ (see Lemma \ref {   centre} and Section \ref{ssec:k1}). If $c=n$, one has $P_\Pi=B_\Pi$, which is a hypersurface of degree $n-1$ (see Propositions \ref{prop:pal} and \ref{prop:pal2}). If $c=n+1$ (and $n$ is odd), the roles are exchanged: the Palatini lines fill up $\bbP^n$ and form a congruence of order one, whose focal locus is the locus $B_\Pi$ of base lines, of dimension $n-2$ (see Theorems \ref{thm:basic}, \ref{thm:prec} and \ref{thm:main}).}
\end{remark}

\part[Rational linear sections]{{Rational linear sections}}\label{part:rat}

 The following theorem is due to Severi \cite{Sev} for the case $k=1$ and to Roth \cite{Ro} for any $k$.

\begin{theorem}\label{thm:mainrat} The general linear section of $\mathbb G(k,n)$ of codimension $c\leq n-k$ is rational. 
\end{theorem}

In the next sections we will recall, for the reader's convenience, the proof of this theorem and we will make some remarks about it. {We first consider the general case, and then, in Section \ref{ssec:k1}, the case $k=1$ of Grassmannians of lines, in which more can be said.}

\section{The general case}\label{ssec:k}

\subsection{The proof in the general case} {Let $k\geq 1$.}

Let $\Pi$ be a general linear subspace of $ \bbP^{N(k,n)}$ of codimension $n-k$ and let $X_\Pi$ be the intersection of $\Pi$ with $\mathbb G(k,n)$, that is a smooth, irreducible, projective variety of dimension $k(n-k)$. The variety $X_\Pi$ parametrizes a $(k(n-k))$--dimensional family $\mathcal X_\Pi$ of linear spaces of dimension $k$.

\begin{lemma}\label{lem:uniqo} If $P$ is a general linear space of dimension $k-1$ of $\bbP^n$ there is a unique linear space of $\mathcal X_\Pi$ containing $P$. 
\end{lemma}

\begin{proof} Let $\Omega$ be a general hyperplane section of $\mathbb G(k,n)$ (called a \emph{$k$--linear complex} of $\bbP^n$). There is a rational map 
$$
\omega: \mathbb G(k-1,n)\dasharrow {\check \bbP^n}, 
$$
such that given a general $[P]\in \mathbb G(k-1,n)$ one has that $P\subset \omega(P)$, and the set of linear spaces of dimension $k$ represented by points of $\Omega$ containing $P$ is the set of all   linear spaces of dimension $k$ containing $P$ and contained in the hyperplane $\omega(P)$. 

Since $X_\Pi$ is the intersection of $n-k$ linearly independent hyperplane sections $\Omega_1,\ldots,\\ \Omega_{n-k}$ of $\mathbb G(k,n)$, given a general $[P]\in \mathbb G(k-1,n)$, the set of linear spaces of dimension $k$ represented by points of $X_\Pi$ containing $P$ is the set of all the linear spaces of dimension $k$ containing $P$ and contained in the intersection $\omega_1(P)\cap \cdots \cap \omega_{n-k}(P)$. Since this intersection is a linear space of dimension $k$, this concludes the proof.  \end{proof}

{For $k=1$, Lemma \ref{lem:uniqo} is Lemma \ref { centre}, (iii), with $\mathcal L=\mathcal L_\Pi$, and its proof is the one of Lemma \ref{lem:hyp}.}

Theorem \ref {thm:mainrat} for $c=n-k$ is an immediate consequence of this lemma.

\begin{proof}[Proof of Theorem \ref {thm:mainrat} for $c=n-k$] Let $\pi$ be a general hyperplane of $\bbP^n$. We identify $\mathbb G(k-1,n-1)$ with the set of all linear spaces of dimension $k-1$ contained in $\pi$. The map
$$
\rho_{\Pi,\pi}: X_\Pi \dasharrow \mathbb G(k-1,n-1), \quad [P] \mapsto P\cap \pi
$$
is birational by  Lemma \ref {lem:uniqo}. 
\end{proof}

Following Roth in \cite {Ro}, one can be more precise about the map $\rho_{\Pi,\pi}$. First of all we have the following lemma, whose proof is immediate:

\begin{lemma}\label{lem:indetero} The indeterminacy locus of $\rho_{\Pi,\pi}$ is 
$$
X_\Pi\cap \mathbb G(k,n-1)= \Pi\cap \mathbb G(k,n-1){=:\mathcal R_{\Pi,\pi}}
$$
that has dimension $k(n-k-1)-1$, where $\mathbb G(k,n-1)$ is identified with the Grassmannian of linear spaces of dimension $k$ contained in $\pi$. 
\end{lemma}

Even more precisely, one has that

\begin{proposition}\label{prop:projo} The map $\rho_{\Pi,\pi}$ is the linear projection of $X_\Pi$ from the linear span of $ \Pi\cap \mathbb G(k,n-1)$ (that has dimension {$N(k,n-1)-(n-k)$}) to $\mathbb G(k-1,n-1)$.
\end{proposition}

\begin{proof} Analogous to the proof of {Lemma \ref {lem:projR}}. Indeed, we claim that the map  $\rho_{\Pi,\pi}$ is linear in the Pl\"ucker coordinates of $X_\Pi$. In fact, we may assume that the hyperplane $\pi$ has equation $x_0=0$. If $P$ is a general space of $X_\Pi$, we may assume that it has Pl\"ucker coordinates given by the order $k+1$ determinants of a matrix of the form

\begin{equation}\label{eq:min}
\left ( 
\begin{matrix}
1&0&a_{02}&a_{03}&\ldots&\ldots &\ldots&\ldots&a_{0n}\\
1&a_{11}&0&a_{13}&\ldots&\ldots &\ldots&\ldots&a_{1n}\\
.&.&.&.&\ldots &\ldots&&\ldots&.\\
1&a_{k1}&a_{k2}&a_{k3}&\dots&a_{kk}&0&\ldots &a_{kn}\\
\end{matrix}
\right).
\end{equation}

The Pl\"ucker coordinates of the intersection of  $P$ with $\pi$ are the minors of order $k$ of the matrix

$$
\left (
\begin{matrix}
a_{11}&{-a_{02}}&a_{13}-a_{03}&\ldots&\ldots &\ldots&\ldots&a_{1n}-a_{0n}\\
.&.&.&.&\ldots &\ldots&\ldots&.\\
a_{k1}&a_{k2}-a_{02}&a_{k3}-a_{03}&\dots&a_{kk}-a_{0k}&-a_{0,k+1}&\ldots &a_{kn}-a_{0n}\\
\end{matrix}
\right)
$$
and these minors can be clearly expressed as linear combinations of the maximal minors of the matrix 
\eqref {eq:min}, which proves our claim. The claim implies the assertion. \end{proof}

As a sanity check we see that 
the assertion of Proposition {\ref {prop:projo}} fits with the dimensions of the ambient spaces. Indeed

$$
\dim(\Pi)-\dim (\langle \Pi\cap \mathbb G(k,n-1)\rangle)-1=N(k,n)-N(k,n-1)-1={N(k-1,n-1)}.
$$

{For $k=1$, Lemma \ref{lem:indetero} and Proposition \ref{prop:projo} are the case $c=n-1$ of Lemma \ref{lem:projR}.}

As an immediate consequence we have the: 

\begin{proof}[Proof of Theorem \ref {thm:mainrat} for $k$ general] We proved already the theorem if $c=n-k$. Now assume that $c<n-k$. Then fix $\pi$ a general hyperplane of $\bbP^n$, identify $\mathbb G(k,n-1)\subset \mathbb G(k,n)$ with the Grassmannian  of linear spaces of dimension $k$ contained in $\pi$ and consider the subset $\mathcal R_{\Pi', \pi}\subset X_\Pi$ that is the intersection of $\mathbb G(k,n-1)$ with a general linear space $\Pi'$ of codimension $n-k$ contained in $\Pi$. Clearly the projection of $X_\Pi$ from the span of $\mathcal R_{\Pi', \pi}$ is a birational map of $X_\Pi$ to a variety $W$ of dimension $(k+1)(n-k)-c$ whose general section with a linear space of {codimension} $n-k-c$ is  a $\mathbb G(k-1,n-1)$. {Since Grassmannians different from $\mathbb G(1,3)$ are not extendable (see \cite {DF, Ba})}, $W$ is a cone {with vertex a linear space of dimension $n-k-c-1$} over a {$\mathbb G(k-1,n-1)$}, and therefore it is rational, {unless $k=2, n=4$, in which case $W$ is a quadric and therefore it is rational}. 
\end{proof}

\subsection{The inverse map for $c=n-k$}\label{ssec:invk} Let again $\Pi$ {be} a general linear subspace of codimension $c=n-k$ in $\bbP^{N(k,n)}$ and let $X_\Pi$ be, as usual, the intersection of $\mathbb G(k,n)$ with $\Pi$.  The indeterminacy locus $Z_{\Pi,\pi}$ of the inverse of the map 
$$
\rho_{\Pi,\pi}: X_\Pi \dasharrow \mathbb G(k-1,n-1)
$$
introduced above is the set of all linear spaces of dimension $k-1$ contained in $\pi$, such that there are infinitely many linear spaces of $\mathcal X_\Pi$ containing them. More generally we can consider the locus $Z_\Pi\subset \mathbb G(k-1,n)$ of all linear spaces of dimension $k-1$ contained in $\bbP^n$, such that there are infinitely many linear spaces of $\mathcal X_\Pi$ containing them. This is the analogue of the focal locus in the case $k=1$. One has 
$$
Z_{\Pi,\pi}=Z_\Pi\cap \mathbb G(k-1,n-1)
$$
where $\mathbb G(k-1,n-1)$ is the Grassmannian of linear spaces of dimension $k-1$ contained in $\pi$. 

The following result  has been proved by Roth (see \cite [\S 5]{Ro}). Here we reproduce the (rather terse) proof of Roth for the reader's convenience:

\begin{proposition}\label{prop:n-k} The inverse map $\rho_{\Pi,\pi}^{-1}: \mathbb G(k-1,n-1)\dasharrow X_\Pi$ is defined by a linear subsystem of $|\mathcal O_{\mathbb G(k-1,n-1)}(n-k)|$ and precisely by a linear subsystem $\mathcal S$ of {the linear system}
$$|\mathcal I_{Z_{\Pi,\pi}|\mathbb G(k-1,n-1)}(n-k)|.$$
\end{proposition} 

\begin{proof} To prove that $\rho_{\Pi,\pi}^{-1}$ is defined by a linear subsystem of $|\mathcal O_{\mathbb G(k-1,n-1)}(n-k)|$ it suffices to prove that the general hyperplane section of $X_\Pi$ is mapped via $\rho_{\Pi,\pi}$ to the section $D$ of $\mathbb G(k-1,n-1)$ by a hypersurface of degree $n-k$. We keep the notation introduced above. 

The general hyperplane section of $X_\Pi$ is of the form $X_{\Pi'}$, where $\Pi'$ is a general linear subspace $\Pi'$ of $\bbP^{N(k,n)}$ of codimension $n-k+1$.  The degree $m$ of a hypersurface cutting out $D$ on $\mathbb G(k-1,n-1)$ is the intersection number of $D$ with a general line $\ell$ contained in  $\mathbb G(k-1,n-1)$. A general line $\ell$ in $\mathbb G(k-1,n-1)$ is a pencil of  linear spaces of dimension $k-1$ passing through a given linear subspace $\sigma$ of dimension $k-2$ contained in a fixed linear space of dimension {$k$} contained in $\pi$. Consider then the family $\mathcal X_{\Pi',\sigma}$ of all linear spaces in $X_{\Pi'}$ that contain $\sigma$, i.e., containing $k-1$ general points of $\sigma$ {(if $k=1$, then $\sigma=\emptyset$ and $\mathcal X_{\Pi',\sigma}$ is the family of all lines of $X_{\Pi'}$)}. A standard count of parameters shows that $\mathcal X_{\Pi',\sigma}$ has dimension $n-k-1$, and therefore the $k$--dimensional linear spaces of $\mathcal X_{\Pi',\sigma}$ fill up a hypersurface 
$F_{\Pi',\sigma}$ of $\bbP^n$. It is clear that $m=\deg(F_{\Pi',\sigma})$. To compute this degree $m$, fix a general linear space $\tau$ of dimension $n-k+1$ of $\bbP^n$. The intersection of $\tau$ with $F_{\Pi',\sigma}$ is a hypersurface $F_{\Pi',\sigma, \tau}$ of dimension $n-k$ and degree $m$ in $\tau$, that is described by the lines common to $n - k + 1$  line complexes of $\tau$. By {the computation in the proof of Lemma \ref {lem:roth} (see also Propositions \ref{prop:pal} and \ref{prop:pal2})}, we have that the degree $m$ equals $n-k$, as wanted. \end{proof} 

We can describe  $Z_\Pi$ (and therefore $Z_{\Pi,\pi}$) {as a degeneracy locus, extending to any $k$ the construction of \S \ref{ssec:omega}}. Indeed, {we have} $n-k$ independent hyperplanes $H_1,\ldots H_{n-k}$ that cut out $\Pi$. Each of these hyperplanes $H$ gives rise to a map 
$$
\beta_H: E(-1)\longrightarrow \mathcal O_{\mathbb G(k-1,n)}
$$
where $E$ is the universal quotient bundle over $\mathbb G(k-1,n)$, that has rank $n-k+1$. Hence we have a map
$$
\beta_\Pi:  E(-1)\longrightarrow \mathcal O^{n-k}_{\mathbb G(k-1,n)}
$$
and dually a map
$$
w_\Pi: \mathcal O^{n-k}_{\mathbb G(k-1,n)} \longrightarrow {\check E(1)}
$$
{(for $k=1$ one has $E=T_{\bbP^n}(-1)$, hence ${\check E(1)}\cong \Omega_{\bbP^n}(2)$, $\beta_H=\alpha_H$ and $w_\Pi=w_{\mathcal L_\Pi}$, see \S \ref{ssec:omega}).}
The variety $Z_\Pi$  is the locus where this map drops rank, and it has codimension 2 in $\mathbb G(k-1,n)${, for $\Pi$ general (see Lemma \ref{lem:degenv} below)}.

\begin{proposition}\label{prop:class} The class of $Z_\Pi$ in the Chow ring of $\mathbb G(k-1,n)$ is
$$
{{n-k}\choose 2} \sigma_1^2+\sigma_2
$$
where $\sigma_1$ is the hyperplane section class and $\sigma_2$ is the codimension 2 class of all linear spaces of dimension $k-1$ intersecting a fixed space of dimension $n-k-1$. 
\end{proposition}

\begin{proof} One has 
$$
c(E)=1+\sigma_1+\sigma_2+ \cdots
$$
(see \cite [\S 5.6.2]{EH}). Hence
$$
c({\check E})=1-\sigma_1+\sigma_2- \cdots.
$$
Then the assertion follows by \cite [Prop. 5.17]{EH}. \end{proof}

{
\begin{lemma}\label{lem:degenv} Let $\Pi$ be general {of codimension $n-k$}. Then the scheme $Z_\Pi$ defined by the maximal minors of $\beta_\Pi$ is either empty or a reduced, Cohen--Macaulay subscheme of pure codimension 2 of $\mathbb G(k-1,n)$, and there is an exact sequence
\begin{equation}\label{eq:HBG}
0\longrightarrow \mathcal O^{n-k}_{\mathbb G(k-1,n)}\stackrel{w_\Pi}\longrightarrow {\check E(1)} \longrightarrow \mathcal I_{Z_\Pi|\mathbb G(k-1,n)}(n-k)\longrightarrow 0.
\end{equation}
The same holds on $\mathbb G(k-1,n-1)$: the scheme $Z_{\Pi,\pi}=Z_\Pi\cap \mathbb G(k-1,n-1)$, defined by the maximal minors of the restriction of $\beta_\Pi$, is either empty or a reduced, Cohen--Macaulay subscheme of pure codimension 2 of $\mathbb G(k-1,n-1)$, and there is an exact sequence
\begin{equation}\label{eq:HBGp}
0\longrightarrow \mathcal O^{n-k}_{\mathbb G(k-1,n-1)}\longrightarrow {\check E_{|\mathbb G(k-1,n-1)}(1)}\longrightarrow \mathcal I_{Z_{\Pi,\pi}|\mathbb G(k-1,n-1)}(n-k)\longrightarrow 0.
\end{equation}
\end{lemma}

\begin{proof} Set $r=n-k$, and let $Y$ be either $\mathbb G(k-1,n)$ or $\mathbb G(k-1,n-1)$. For ${\bf W}\in Y$, i.e., ${\bf W}\subset {\bf V}$ a subspace of dimension $k$, where $\bbP^n=\bbP({\bf V})$, the fibre of $\check E(1)$ at ${\bf W}$ is $({\bf V}/{\bf W})^\vee \otimes(\wedge^k{\bf W})^\vee$, and the linear map
$$
\wedge^{k+1}\check {\bf V}\longrightarrow ({\bf V}/{\bf W})^\vee \otimes(\wedge^k{\bf W})^\vee,\quad h\mapsto \big(\bar u\otimes \omega\mapsto h(\omega\wedge u)\big),
$$
is surjective, because $\wedge^k{\bf W}\otimes ({\bf V}/{\bf W})\to \wedge^{k+1}{\bf V}$, $\omega\otimes \bar u\mapsto \omega\wedge u$, is injective. The morphism $\beta_\Pi$ is determined by $(H_1,\ldots, H_r)\in M:=(\wedge^{k+1}\check {\bf V})^{r}$, and for every ${\bf W}\in Y$ the evaluation map $M\to {\rm Hom}(E(-1)_{\bf W},\bbC^r)$ is surjective. Let $\Delta_j\subset {\rm Hom}(\bbC^{r+1},\bbC^r)$ be the variety of matrices of rank at most $j$: $\Delta_{r-1}$ is irreducible of codimension 2 and its singular locus is $\Delta_{r-2}$, which has codimension 6. Hence the incidence variety $\{({\bf W},\beta)\in Y\times M: {\rm rk}(\beta_{\bf W})\leq r-1\}$ is irreducible of codimension 2 in $Y\times M$, and it is smooth off the locus where ${\rm rk}(\beta_{\bf W})\leq r-2$, which has codimension 6. By generic smoothness and by dimension count, for general $\Pi$ the degeneracy locus of $\beta_{\Pi|Y}$ is either empty or of pure codimension 2 in $Y$, and it is smooth off a closed subset of codimension at least 6 in $Y$ (this is a Bertini type theorem for degeneracy loci, see \cite[\S 4.1]{Ban}). Since it is locally defined by the maximal minors of a $r\times (r+1)$ matrix and it has the expected codimension 2, it is Cohen--Macaulay (see \cite{EN}), hence it is reduced, being generically smooth. Finally, $w_{\Pi|Y}$ is locally given by a $(r+1)\times r$ matrix whose maximal minors define the degeneracy locus, which has codimension 2, so Lemma \ref{lem:HB} gives the exact sequences \eqref{eq:HBG} and \eqref{eq:HBGp}; the twist is $\det({\check E(1)})\cong \mathcal O(n-k)$, because $c_1({\check E(1)})=(n-k)\sigma_1$.
\end{proof}
}

\begin{theorem}\label{thm:conj}  The map $\rho_{\Pi,\pi}^{-1}$ is defined by the full linear system  $|\mathcal I_{Z_{\Pi,\pi}|\mathbb G(k-1,n-1)}(n-k)|$. 
\end{theorem}

\begin{proof} It suffices to prove that
$$
h^0(\mathbb G(k-1,n-1), \mathcal I_{Z_{\Pi,\pi}|\mathbb G(k-1,n-1)}(n-k))={{n+1}\choose {k+1}}-n+k.
$$
{Indeed, by Proposition \ref{prop:n-k}, the linear system $\mathcal S$ defining $\rho^{-1}_{\Pi,\pi}$ is contained in the linear system of the statement, and by Lemma \ref{lem:degenv} the scheme $Z_{\Pi,\pi}$ is reduced, so there is no ambiguity about its scheme structure. Moreover $\mathcal S$ is the pull--back via $\rho^{-1}_{\Pi,\pi}$ of the linear system of hyperplane sections of $\Pi$; since $\rho^{-1}_{\Pi,\pi}$ is dominant onto $X_\Pi$, which is non--degenerate in $\Pi$, the vector space corresponding to $\mathcal S$ has dimension $\dim(\Pi)+1={{n+1}\choose {k+1}}-n+k$.}

Note that we have the exact sequence
$$
0\longrightarrow \mathcal O^{n-k}_{\mathbb G(k-1,n)} \longrightarrow {\check E(1)}\longrightarrow \mathcal I_{Z_\Pi|\mathbb G(k-1,n)}(n-k)\longrightarrow 0
$$
{by Lemma \ref{lem:degenv}}. {By the same lemma} we have {also} the exact sequence
$$
0\longrightarrow \mathcal O^{n-k}_{\mathbb G(k-1,n-1)} \longrightarrow{\check E_{|\mathbb G(k-1,n-1)}(1)}\longrightarrow \mathcal I_{Z_{\Pi,\pi}|\mathbb G(k-1,n-1)}(n-k)\longrightarrow 0
$$
where, as usual, $\mathbb G(k-1,n-1)$ is identified with the Grassmannian of linear spaces of dimension $k-1$ contained in the hyperplane $\pi$. From this sequence{, since $H^1(\mathbb G(k-1,n-1),\mathcal O_{\mathbb G(k-1,n-1)})=0$,} we see that
$$
\begin{aligned}
&h^0(\mathbb G(k-1,n-1), \mathcal I_{Z_{\Pi,\pi}|\mathbb G(k-1,n-1)}(n-k))=\\
&\qquad\qquad =h^0(\mathbb G(k-1,n-1), {\check E_{|\mathbb G(k-1,n-1)}(1)})-n+k.
\end{aligned}
$$
So we need to show that
$$
h^0(\mathbb G(k-1,n-1),  {\check E_{|\mathbb G(k-1,n-1)}(1)})={{n+1}\choose {k+1}}. 
$$

Consider the dual of the tautological sequence of $\mathbb G(k-1,n)$ restricted to $\mathbb G(k-1,n-1)$
\begin{equation}\label{eq:res0}
0\longrightarrow   {\check E_{|\mathbb G(k-1,n-1)}(1)}\longrightarrow \mathcal O^{n+1}_{\mathbb G(k-1,n-1)}(1)\longrightarrow  {\check S_{|\mathbb G(k-1,n-1)}(1)} \longrightarrow 0
\end{equation}
where $S$ is the tautological bundle of $\mathbb G(k-1,n)$, whose restriction to $\mathbb G(k-1,n-1)$ is clearly the tautological bundle $S'$ of $\mathbb G(k-1,n-1)$. We have also the dual of the tautological sequence of $\mathbb G(k-1,n-1)$, namely
\begin{equation}\label{eq:res}
0\longrightarrow  {\check E'(1)}\longrightarrow \mathcal O^{n}_{\mathbb G(k-1,n-1)}(1)\longrightarrow {\check S'(1)} \longrightarrow 0
\end{equation}
where $E'$ is the universal quotient bundle on $\mathbb G(k-1,n-1)$. By a simple application of the Borel--Weil--Bott Theorem (see \cite [Cor.  4.1.9]{We}{)} one has  
$$
h^1(\mathbb G(k-1,n-1), {\check E'(1)})=0,
$$
{indeed we have ${\check E'(1)}\cong \wedge^{n-k-1}E'$, whose only non--zero cohomology group is $H^0$. Hence}
the sequence \eqref {eq:res} is exact at $H^0$ level.
Moreover the map in \eqref{eq:res0} $$ \mathcal O^{n+1}_{\mathbb G(k-1,n-1)}(1)={\check {\bf V}}\otimes \mathcal O_{\mathbb G(k-1,n-1)}(1) \to {\check S'(1)}$$  factors through the restriction map $$ {\check {\bf V}}\otimes \mathcal O_{\mathbb G(k-1,n-1)}(1)\to {\check {\bf V}'}\otimes\mathcal O_{\mathbb G(k-1,n-1)}(1)=\mathcal O^{n}_{\mathbb G(k-1,n-1)}(1), $$ which is surjective on global sections. {Composing it with the map in \eqref{eq:res}, it follows that} \eqref{eq:res0} is exact at $H^0$ level. Hence
$$
h^0(\mathbb G(k-1,n-1), {\check E_{|\mathbb G(k-1,n-1)}
(1)})=(n+1){n\choose k}-h^0(\mathbb G(k-1,n-1),{\check S'(1)}).
$$

Another application of the Borel--Weil--Bott Theorem shows that 

$$
h^0(\mathbb G(k-1,n-1),{\check S'(1)})=k {{n+1}\choose {k+1}}.
$$

{{This just follows from  ${\check S'(1)}={\check S'}\otimes\wedge^k{\check S'}\cong \mathbb S_{(2,1^{k-1})}{\check S'}$  {(we denote by $\mathbb S_\lambda$ the Schur functor determined by a partition $\lambda$) by Pieri's rule}, }because $\wedge^{k+1}{\check S'}=0$. { Then } $H^0({\mathbb G(k-1,n-1),} {\check S'(1)})\cong \mathbb S_{(2,1^{k-1})}{\check {\bf V}'}$, which is the kernel of the surjective map $\check {\bf V}'\otimes\wedge^k \check {\bf V}'\to\wedge^{k+1}\check {\bf V}'$, and has dimension $n{n\choose k}-{n\choose {k+1}}=k{{n+1}\choose {k+1}}$)}, so that

$$
h^0(\mathbb G(k-1,n-1), {\check E_{|\mathbb G(k-1,n-1)}(1)})=(n+1){n\choose k}-k {{n+1}\choose {k+1}}={{n+1}\choose {k+1}}
$$

as wanted.  \end{proof}

\begin{remark}\label{rem:oil} (a) It should be remarked that Roth in \cite [\S 5]{Ro} takes the statement of Theorem  \ref {thm:conj} for granted. 

(b) In the same way we proved Theorem \ref {thm:conj}, one can show that
$$
h^0(\mathbb G(k-1,n-1), \mathcal I_{Z_{\Pi,\pi}|\mathbb G(k-1,n-1)}(n-k-1))=1
$$
as expected.
{Indeed, twisting \eqref{eq:HBGp} by $\mathcal O_{\mathbb G(k-1,n-1)}(-1)$, and since this line bundle has
vanishing $h^0$ and $h^1$ by Kodaira vanishing, we get $$h^0(\mathcal I_{Z_{\Pi,\pi}|\mathbb G(k-1,n-1)}(n-k-1))={h^0(\mathbb G(k-1,n-1), \check E_{|\mathbb G(k-1,n-1)})}.$$ {Therefore, by} the dual of the tautological sequence, ${H^0(\mathbb G(k-1,n-1), \check E_{|\mathbb G(k-1,n-1)})}$ is the kernel of the restriction map {$\check {\bf V}\to \check {\bf V}'=H^0(\mathbb G(k-1,n-1), \check S')$,} which is spanned by the equation of $\pi$.}
\end{remark}

\begin{example}[Cfr. \cite{Ro}, \S 5]\label{ex:exo} Let us consider the case $n=4, k=2$ and $c=n-k=2$ as usual. Here $X_\Pi\subset \bbP^7$ is projected {birationally} to $\mathbb G(1,3)$ from a line $\ell\subset X_\Pi$.
Then $Z_\Pi$ is a subvariety of codimension 2 of  $\mathbb G(1,4)$, so it has dimension 4, and its class is
$$
{\sigma_1^2+\sigma_2}.
$$

The variety $Z_{\Pi,\pi}=Z_\Pi\cap \mathbb G(1,3)$ is a surface in the Klein quadric $\mathbb G(1,3)$. Now $\sigma_1$ restricted to $\mathbb G(1,3)$ is the hyperplane section of $\mathbb G(1,3)$ and $\sigma_2$ restricted to $\mathbb G(1,3)$ is a plane. So $Z_{\Pi,\pi}$ is a surface of degree 3 in 
$\mathbb G(1,3)$ that is a rational normal scroll contained in a hyperplane section $H$ of  $\mathbb G(1,3)$. {By Theorem \ref{thm:conj}}, the inverse map $\mathbb G(1,3)\dasharrow X_\Pi$ is given by the linear system of quadric {sections} of $\mathbb G(1,3)$ passing through $Z_{\Pi,\pi}$. Via this map the hyperplane section $H$ is contracted to the centre of projection $\ell$.
\end{example}

\section{The case of the Grassmannian of lines}\label{ssec:k1}

\subsection[The birational map]{{The birational map}}

Let $\Pi$ be a general linear subspace of $ \bbP^{N(1,n)}$ of codimension $n-1$ and let $X_\Pi$ be the intersection of $\Pi$ with $\mathbb G(1,n)$, that is a smooth, irreducible, projective variety of dimension $n-1$. We can consider the linear system $\mathcal L_\Pi$ of dimension $n-2$ of line complexes containing $X_\Pi$ and the family $\mathcal X_\Pi$ of base lines for $\mathcal L_\Pi$, i.e., 
$\mathcal X_\Pi$ is the family of lines parametrized by  $X_\Pi$. As usual we denote by $B_\Pi$ the locus of base lines of $\mathcal L_\Pi$. We can also consider the Palatini locus $P_\Pi$ of $\mathcal L_\Pi$, and we note that $\dim(P_\Pi)\leq n-2$.

{Thus $\mathcal L_\Pi$ is a general linear system of line complexes of dimension $n-2$, and we can apply the results of \S \ref{ssec:cong}. In particular, by Lemma \ref{ centre}, through the general point of $\bbP^n$ there passes a unique base line of $\mathcal L_\Pi$ and $B_\Pi=\bbP^n$ (this is the case $k=1$ of Lemma \ref{lem:uniqo}). Hence, if $\pi$ is a general hyperplane of $\bbP^n$, the map
$$
\rho_{\Pi,\pi}: X_\Pi \dasharrow \pi, \quad [\ell] \mapsto \ell\cap \pi
$$
is birational (see Lemma \ref{lem:proj}), which proves Theorem \ref{thm:mainrat} for $k=1$ and $c=n-1$. By Lemma \ref{lem:projR} (with $c=n-1$), the indeterminacy locus of $\rho_{\Pi,\pi}$ is $\mathcal R_{\Pi,\pi}=\Pi\cap \mathbb G(1,n-1)$, which has dimension $n-3$, and $\rho_{\Pi,\pi}$ is the linear projection of $X_\Pi$ from $\langle \mathcal R_{\Pi,\pi}\rangle$, which has dimension $N(1,n-1)-(n-1)$, to $\pi\cong \bbP^{n-1}$.

If $\Pi$ has codimension $c<n-1$, the variety $W$ in the proof of Theorem \ref{thm:mainrat} is a cone with vertex a linear space of dimension $n-2-c$ over $\mathbb G(0,n-1)=\pi\cong \bbP^{n-1}$, i.e., it is a linear space of dimension $2n-2-c$. {Now the projection} of $X_\Pi$ from the span of $\mathcal R_{\Pi',\pi}$, where $\Pi'$ is a general subspace of codimension $n-1$ contained in $\Pi$, is a birational map of $X_\Pi$ to $\bbP^{2n-2-c}$.} Again, as a sanity check we see that 
$$
\dim(\Pi)-\dim (\langle \mathcal R_{\Pi', \pi} \rangle)-1=(N(1,n)-c)-(N(1,n-1)-(n-1))-1=2n-2-c.
$$

\subsection{The focal locus} We now focus again on the case $c=n-1$. According to {\S \ref {ssec:foci}}, one can consider the focal locus $Z_\Pi$ of the family $\mathcal X_\Pi$ in $\bbP^n$. Recall that for the general $[\ell]\in \mathcal X_\Pi$, the focal scheme on the line $\ell$ has length $n-1$. The focal locus  $Z_\Pi$ can be identified with the locus of points $p\in \bbP^n$ such that there are infinitely many lines of $\mathcal X_\Pi$ passing through $p$ {(see Lemma \ref{lem:focinf})}. Of course one has $\dim(Z_\Pi)\leq n-2$.

{By Lemma \ref{ centre}, $Z_\Pi$ coincides with the Palatini locus $P_\Pi=V_{\mathcal L_\Pi}$ of $\mathcal L_\Pi$, hence it is the degeneracy locus of the morphism 
$$
w_\Pi:=w_{\mathcal L_\Pi}: \mathcal O_{\bbP^n}^{n-1}\longrightarrow \Omega_{\bbP^n}(2)
$$
defined by $n-1$ complexes spanning $\mathcal L_\Pi$ (see \S \ref{ssec:omega} and Remark \ref{rem:lit}). This is the case $k=1$ of the morphism $w_\Pi$ considered in \S \ref{ssec:invk}, because for $k=1$ one has ${\check E(1)}\cong \Omega_{\bbP^n}(2)$.}

\begin{theorem}\label{thm:precis} One has $\dim(Z_\Pi)= n-2$ and 
$$
\deg(Z_\Pi)={{n+1}\choose 2}-2n+2.
$$
Through a general point of $Z_\Pi$ pass a one--dimensional family of  base lines.
\end{theorem} 

\begin{proof} {By Lemmas \ref{lem:3} and \ref{ centre}, $Z_\Pi=V_{\mathcal L_\Pi}$ has dimension $n-2$ and degree $\frac{n(n-3)}2+2={{n+1}\choose 2}-2n+2$. The degree can also be computed with Porteous' formula: by Proposition \ref{prop:class} {it is ${{n-1}\choose 2}+1$ for $k = 1$}, which is the same number. The final assertion is Lemma \ref{ centre}, (v).} \end{proof}

Note that we have the exact sequence
\begin{equation}\label{eq:es}
0\longrightarrow  \mathcal O_{\bbP^n}^{n-1}(1-n)\longrightarrow \Omega_{\bbP^n}(3-n)\longrightarrow\mathcal I_{Z_\Pi|\bbP^n}\longrightarrow 0.
\end{equation}
{This is the case $k=1$ of \eqref{eq:HBG} (see Lemma \ref{lem:degenv}), twisted by $\mathcal O_{\bbP^n}(1-n)$.}

Since $h^0(\bbP^n,\Omega_{\bbP^n}(t))=0$, if $t\leq	1$, we have that 
$$h^0(\bbP^n,\mathcal I_{Z_\Pi|\bbP^n}(t))=0, \,\, \text{ if}\,\,  t\leq n-2,$$
and a {straightforward}   computation shows that

$$
h^0(\bbP^n,\mathcal I_{Z_\Pi|\bbP^n}(n-1))={{{n+1}\choose 2}-n+1}.
$$

\begin{example}	\label{ex:ex0} If $n=3$, $Z_\Pi$ consists of a pair of skew lines. If $n=4$, $Z_\Pi$ is the projection of the Veronese surface of degree 4 in $\bbP^5$ to $\bbP^4$ (see \cite [\S 7-8]{Cas}{, \cite[Thm. 2.1.6 and Ex. 10.2.19]{Dol}}). For $n=5$,  $Z_\Pi$ is  the famous \emph{Palatini scroll} (see, e.g., \cite {M}{, \cite[Ex. 10.2.21]{Dol}}). 
\end{example}

\subsection{The inverse map}

Next we can look at the inverse map
$$
\rho_{\Pi,\pi}^{-1}: \pi \dasharrow X_\Pi
$$
that sends the general point $p\in \pi$ to the point $[\ell]\in X_\Pi$, where $\ell$ is the unique base line of $\mathcal L_\Pi$ passing through $p$. 

The indeterminacy locus of $\rho_{\Pi,\pi}^{-1}$ is the locus of points $p\in \pi$ such that there are infinitely many base lines of $\mathcal L_\Pi$ passing through $p$. Hence the  indeterminacy locus of $\rho_{\Pi,\pi}^{-1}$ is the intersection scheme $Z_{\Pi,\pi}:=\pi\cap Z_\Pi$.

{The following is the case $k=1$ of Proposition \ref{prop:n-k}, Theorem \ref{thm:conj} and Remark \ref{rem:oil}, (b).}

\begin{proposition}\label{prop:map} The map $\rho_{\Pi,\pi}^{-1}$ is defined by the linear system $|\mathcal I_{Z_{\Pi,\pi}|\pi}(n-1)|$. In particular
\begin{equation}	\label{eq:exp}
h^0({\pi},\mathcal I_{Z_{\Pi,\pi}|\pi}(n-1))={{n}\choose 2}+1, \quad h^0({\pi},\mathcal I_{Z_{\Pi,\pi}|\pi}(n-2))=1.
\end{equation}
\end{proposition}

\begin{proof} {For $k=1$ one has $\mathbb G(k-1,n-1)=\pi$, $n-k=n-1$ and ${{n+1}\choose {k+1}}-n+k={{n+1}\choose 2}-n+1={n\choose 2}+1$. Hence the first assertion and the first equality in \eqref{eq:exp} follow from Proposition \ref{prop:n-k} and Theorem \ref{thm:conj}, and the second equality in \eqref{eq:exp} follows from Remark \ref{rem:oil}, (b).} \end{proof}

Moreover, since $\rho_{\Pi,\pi}$ is a projection, 
$|\mathcal I_{Z_{\Pi,\pi}|\pi}(n-1)|$ has to contain $|{\mathcal O_{\pi}(1)}|$ plus a fixed part, that is the unique divisor in $|\mathcal I_{Z_{\Pi,\pi}|\pi}(n-2)|$.

\begin{remark}\label{rem:contr} The unique hypersurface of degree $n-2$ containing $Z_{\Pi,\pi}$ is contracted by $\rho_{\Pi,\pi}^{-1}$  to $\Pi\cap \mathbb G(1,n-1)$. 
\end{remark}

\begin{example}\label{ex:ex} (a) For $n=3$, $X_\Pi$ is a smooth quadric surface $Q$ in $\Pi=\bbP^3$, and $\rho_{\Pi,\pi}$ is the stereographic projection of $Q$ from {the point }$p$. Accordingly 
$Z_{\Pi,\pi}$ is a pair of distinct points in $\pi=\bbP^2$, and $\rho_{\Pi,\pi}^{-1}$ is given by the linear system of conics through $Z_{\Pi,\pi}$. The line joining the two points is contracted by $\rho_{\Pi,\pi}^{-1}$ to the centre of projection $p$.

(b) For $n=4$ we have that  $X_\Pi$ is a smooth threefold of degree 5 in $\Pi=\bbP^6$, and $\rho_{\Pi,\pi}$ is the projection of $X_\Pi$ from the plane spanned by a conic $\Gamma=\Pi\cap \mathbb G(1,3)$. Then  
$Z_{\Pi,\pi}$ is a smooth rational curve of degree 4 in  $\pi\cong \bbP^3$, and $\rho_{\Pi,\pi}^{-1}$ is given by the linear system of cubics through $Z_{\Pi,\pi}$. The quartic  $Z_{\Pi,\pi}$ sits on a unique quadric $Q$ that is contracted to the conic $\Gamma$ by $\rho_{\Pi,\pi}^{-1}$.

(c) For $n=5$ we have that  $X_\Pi$ is a smooth 4--fold of degree 14 in $\Pi\cong \bbP^{10}$, and $\rho_{\Pi,\pi}$ is the projection of $X_\Pi$ from the {linear space of dimension 5} spanned by the degree 5 Del Pezzo surface $S=\Pi\cap \mathbb G(1,4)$. Then
$Z_{\Pi,\pi}$ is a smooth rational surface of degree 7 in  $\pi\cong \bbP^4$, and $\rho_{\Pi,\pi}^{-1}$ is given by the linear system of quartics through $Z_{\Pi,\pi}$. The surface $Z_{\Pi,\pi}$ sits on a unique cubic hypersurface that is contracted to the Del Pezzo surface $S$ by $\rho_{\Pi,\pi}^{-1}$. 

The surface $Z_{\Pi,\pi}$ is a general hyperplane section of the Palatini scroll and it is well {known} to be represented on the plane by {the} linear system of plane quintic curves with two general base points of multiplicity 2 and ten general simple base points. 
\end{example} 

\begin{remark}\label{rem:gen} (a) If $\Pi$ is general of codimension $c<n-1$, according to the proof of Theorem \ref {thm:mainrat} {(see also \S \ref{ssec:k1})}, the projection
$$
\sigma_{\Pi, \mathcal R_{\Pi', \pi}}: X_\Pi\dasharrow \bbP^{2n-2-c}
$$
from $\mathcal R_{\Pi', \pi}$ is birational (see the proof of Theorem \ref {thm:mainrat}{}, here $\Pi'$ is a general subspace of codimension $n-1$ contained in $\Pi$). The inverse map $\sigma_{\Pi, \mathcal R_{\Pi', \pi}}^{-1}$ is again defined by a linear system of forms of degree $n-1$ that pass through a codimension 2 subvariety of $\bbP^{2n-2-c}$ that is a suitable extension of $Z_{\Pi',\pi}$. 

(b) This applies also to the case $c=0$ of the whole Grassmannian. For instance $\mathbb G(1,4)$ is birationally projected to $\bbP^6$ from the plane of a general conic $\Gamma$ on it. The inverse map is defined by forms of degree 3 that pass through a general projection $Y$ in $\bbP^6$ of the Segre variety ${\rm Seg}(1,3)\cong \bbP^1\times \bbP^3$ from a general point $p$ in $\bbP^7$.  The variety $Y$ has a double plane because the secant lines to ${\rm Seg}(1,3)$ passing through $p$ form a quadric cone of dimension 3, {whose image is} the double plane in question. The projection of $Y$ from this double plane is a smooth quadric surface in $\bbP^3$, hence $Y$ sits on a unique quadric whose image via the inverse map is the conic  $\Gamma$.

In general $\mathbb G(1,n)$ is birationally projected to $\bbP^{2n-2}$ from {the} span of a variety $Z$ that is  the intersection of a general linear space $\Pi$ of codimension $n-1$ with a $\mathbb G(1,n-1)$ (identified with the lines contained in a general hyperplane $\pi$ of $\bbP^n$). The variety $Z$ has dimension $n-3$ and degree 
$$
\deg(Z)= \deg(\mathbb G(1,n-1))=\frac {(2n-4)!}{(n-1)!(n-2)!}.
$$
The inverse map is defined by forms of degree $n-1$ that pass through a  variety $Y$ of dimension $2n-4$ and degree ${{n+1}\choose 2}-2n+2$ that is a suitable extension of the general hyperplane section of the focal locus $Z_\Pi$ of the family $\mathcal X_\Pi$ with $\Pi$ a general codimension $n-1$ subspace of $\mathbb P^{N(1,n)}$. As far as we know, this representation of Grassmannians $\mathbb G(1,n)$ never appeared in the literature so far.

(c) Simpler representations of Grassmannians have been known for a long time. For example Semple showed in \cite {sem} the existence of representations of Grassmannians $\mathbb G(1,n)$ on $\bbP^{2n-2}$ by a linear system of quadrics. In particular, $\mathbb G(1,n)$ can be projected birationally to $\bbP^{2n-2}$ from the linear span of a $\mathbb G(1,n-2)$ (identified with the Grassmannian of lines contained in a linear subspace of codimension 2 of $\bbP^n$). The inverse map $$\phi_n: \bbP^{2n-2}\dasharrow  \mathbb G(1,n)$$
is defined by the linear system of quadrics in $\bbP^{2n-2}$ that pass through a Segre variety ${\rm Seg}(1, n-2)$ {embedded in a hyperplane $\pi$ of $\bbP^{2n-2}$ via the Segre map}. 

(d) Using this representation one can find special representations of $X_\Pi$ for particular linear spaces $\Pi$. 
For example, $\mathbb G(1,4)$ contains a family of dimension 6 of planes, that are $\mathbb G(1,2)$, identified with the Grassmannians of {lines contained in the planes of} $\bbP^4$. The projection of $\mathbb G(1,4)\subset \bbP^9$ to $\bbP^6$ from one of these planes is birational as we saw above. If $\Pi$ is a general linear space of dimension 6 in $\bbP^9$ it does not contain any plane of $\mathbb G(1,4)$. However there are special linear spaces $\Pi$ of dimension 6 in $\bbP^9$ that contain a plane ${\varpi}$, and the projection of $X_\Pi$ from ${\varpi}$ is a birational map to $\bbP^3$, the inverse map being given by the quadrics passing through three distinct points $p_1,p_2,p_3$ of $\bbP^3$. This is a simpler representation of $X_\Pi$ than the one obtained above for the general $\Pi$, that, as we saw, is defined by cubics. However in this case $X_\Pi$ is no longer smooth, since it has three double points, precisely the images of the three lines pairwise joining the  points $p_1,p_2,p_3$. The plane ${\varpi}$ is mapped to the plane joining the points $p_1,p_2,p_3$, and it contains the three double points of $X_\Pi$. Moreover $X_\Pi$ contains three further planes images of the three points  $p_1,p_2,p_3$ that are blown up. These three planes intersect ${\varpi}$ along a line and intersect each other in a double point. 

(e) {Analogous}  considerations can be made in other similar cases.   For example, consider a {linear} space  $\Pi$ of codimension $1$ in $\bbP^{N(1,n)}$ containing a $\mathbb G(1,n-2)$ (identified with the Grassmannian of lines contained in a linear subspace of codimension 2 of $\bbP^n$). The linear space 
$\Pi$ cuts out on $\mathbb G(1,n)$ a line complex $X_\Pi$ that projects down birationally, from the span of $\mathbb G(1,n-2)$, to a hyperplane $\pi$ of $\bbP^{2n-2}$. The inverse map $\pi\dasharrow X_\Pi$ is determined by the linear system of quadrics in $\pi$ that pass through a hyperplane section of  ${\rm Seg}(1, n-2)$ contained in a linear space of codimension 1 in $\pi$. \end{remark}

   \part{Fano linear sections {of Grassmannians of lines}}\label{part:fano}   

\section{The birationality theorem}\label{sec:fano}

In this section    we maintain      the notation introduced in Section \ref {sec:lc} and    retain     the assumption that $n=2k+1$ is odd and $k\geq 1$. We will focus here on a    specific setting: we fix     a linear space $\Pi$ of  $\bbP^{\frac {n(n+1)}2-1}$ of codimension $n$ hence of dimension
$$
\dim (\Pi)=\frac {n(n-1)}2-1,
$$
   which     intersects $\mathbb G(1,n)$ in a smooth, irreducible variety $X_\Pi$ of codimension $n$, hence of dimension $n-2=2k-1$ and degree
$$
\deg (X_\Pi)=\deg(\mathbb G(1,n))=\frac {(2n-2)!}{n!(n-1)!}.
$$
   In fact     we will often assume that $\Pi$ is a general subspace of codimension $n$ of  $\bbP^{\frac {n(n+1)}2-1}$.  By the Adjunction Formula one has
$$
K_{X_\Pi}\sim \mathcal O_{X_\Pi}(-1),
$$
so $X_\Pi$ is a Fano variety of index 1.  

   Associated with     $\Pi$ and $X_\Pi$, there are the following objects introduced in \S \ref {sserc:ls}:\\
\begin{inparaenum}
\item [$\bullet$] the linear system $\mathcal L_\Pi$,    which     has dimension $n-1=2k$;\\
\item [$\bullet$] the singular locus $\mathcal C_\Pi\subset \mathcal L_\Pi$, that has dimension $n-2=2k-1$ and degree $k+1$    and is     the zero locus of the  Pfaffian  of a general antisymmetric matrix of order $n+1=2(k+1)$;\\
\item [$\bullet$] the locus of  the  base lines $B_\Pi$ in $\bbP^n$, with expected dimension $n-1=2k$;\\
\item [$\bullet$] the Palatini locus $P_\Pi$ in $\bbP^n$, also with expected dimension $n-1=2k$.
\end{inparaenum}

\begin{remark}\label{rem:sing} Notice that, for    a general $\Pi$,     $\mathcal C_\Pi$ is singular in dimension $n-7$ (hence it is smooth only for $n\leq 5$), because ${\check {\mathbb G}(1,n)}$ is singular along the complexes    { whose centre is}      a linear space of dimension at least 3, i.e., in dimension $\frac {n(n+1)}2 -7$. The degree of the singular locus of $\mathcal C_\Pi$ can be computed using the results in \cite {Pal} (or  \cite {HT}), and it is    given by \eqref{eq:degsing}.    
For $n=7$ this tells us that $\mathcal C_\Pi$  has 84 double points. For $n>7$, $\mathcal C_\Pi$ has an irreducible singular locus of the aforementioned degree, with the general singular point a double point. The irreducibility follows from the generality of $\Pi$ and from the irreducibility of the singular locus of the universal  Pfaffian ${\check {\mathbb G}(1,n)}$, because  the set of all complexes    { whose centre is}      a linear space of dimension at least 3 is  clearly   irreducible.

\end{remark} 

\begin{proposition}\label{prop:pal} In the above setup, if $\Pi$ is sufficiently general,    the following properties hold    : \\
\begin{inparaenum}[(i)]
\item the Palatini locus $P_\Pi$ coincides with the locus of base lines $B_\Pi$;\\
\item   a unique Palatini line of the family $\mathcal P_\Pi$ and a unique base line of the family $\mathcal X_\Pi$    pass through    a     general point of $P_\Pi$;\\
\item the Palatini locus $P_\Pi$ has both dimension and degree    equal  to     $n-1=2k$, so it is a hypersurface of degree $n-1$ in $\bbP^n$.\\ 
\end{inparaenum}
\end{proposition}
\begin{proof} 
   
Take a general line $r$ in $\bbP^n$. If a point $p\in r$ lies in the {centre} of a complex $\omega\in \mathcal L_\Pi$, then all lines through $p$ belong to $\omega$; in particular, $r$ does. Hence, $r\cap P_\Pi$ is the set of points of $r$ that lie in the {centre} of some complex of the linear system $\mathcal L'_r\subset \mathcal L_\Pi$ of complexes containing $r$.   This linear system     has dimension $n-2$ and, by the generality of $\Pi$, is general among the linear systems of dimension $n-2$ of complexes containing $r$. By Lemma~\ref{lem:line}, $r$ intersects $P_\Pi$ in $n-1=2k$ points, which proves part (iii) of the assertion.    This also implies that only finitely many Palatini lines of $\mathcal L_\Pi$ pass through a general point of $P_\Pi$.    

Next, we claim that a unique Palatini line of $\mathcal L_\Pi$ passes through a general point $p$ of $P_\Pi$. Indeed, suppose    for contradiction      that two distinct Palatini lines $\ell_\omega$ and $\ell_{\omega'}$ pass through $p$, where $\omega$ and $\omega'$ are two distinct singular line complexes of $\mathcal L_\Pi$. Then $p$, being singular for both $\omega$ and $\omega'$, would be singular for all complexes in the pencil generated by $\omega$ and $\omega'$. Consequently, infinitely many Palatini lines of $\mathcal L_\Pi$ would pass through $p$,   which is a contradiction.      This proves the first assertion of (ii). Let $p$ be again a general point of $P_\Pi$. By the same argument as above, $p$ is in the {centre} of a unique complex of $\mathcal L_\Pi$. Thus, by Lemma~\ref{lem:hyp} (with $\dim(\mathcal L_\Pi)=n-1$), there is a unique base line $b$ of $\mathcal L_\Pi$ passing through $p$. 
Therefore, the general point $p$ of $P_\Pi$ lies in $B_\Pi$. Since {$B_\Pi$ is irreducible of dimension at most $\dim(X_\Pi)+1=n-1$ and contains $P_\Pi$, which has dimension $n-1$}, this proves (i). In particular, only one base line of $\mathcal L_\Pi$ passes through a general point $p$ in $P_\Pi=B_\Pi$, which completes the proof of (ii).
\end{proof}

    Since $P_\Pi$ is the degeneracy locus of a general morphism $\mathcal O_{\bbP^n}^{\oplus n}\to \Omega_{\bbP^n}(2)$ (see {\S \ref{ssec:omega}}), part (iii) of Proposition \ref{prop:pal}, as well as the fact that $P_\Pi$ is a scroll over $\mathcal C_\Pi$ whose rulings are the Palatini lines, can also be deduced from \cite[\S 2]{BM} (see also \cite[\S 1.2]{M}){; note that here $m=n$, so that this case is not covered by \cite{T1, T2} (see Remark \ref{rem:lit})}. {The degree of $P_\Pi$ is also the case $m=n$ of \cite[(10.22)]{Dol}, and for $n=5$ the quartic hypersurface $P_\Pi$ appears in \cite[Ex. 10.2.21]{Dol}.}

As an immediate consequence we have the birationality theorem:
 
\begin{theorem}\label{thm:bir} The Fano variety $X_\Pi$ is  birational to the  Pfaffian  hypersurface $\mathcal C_\Pi$ in 
$\mathcal L_\Pi$. 
\end{theorem}

\begin{proof} Let $H$ be a general hyperplane section of the Palatini locus $P_\Pi$, {   that is, } a hypersurface of degree $n-1$ in $\bbP^{n-1}$. We have a rational map $\alpha: X_\Pi\dasharrow H$, that works as follows. Take a general point $x\in X_\Pi$. It corresponds to  a  base line $\ell_x$ of $\mathcal L_\Pi$. Then $\alpha(x)$ is the intersection point of $\ell_x$ with $H$. This map is    birational by (ii) of Proposition \ref {prop:pal} and Lemma \ref{lem:proj}.      Similarly, there is an analogous rational map $\beta: \mathcal C_\Pi\dasharrow H$ defined in the following way.
 Take a general point $\omega\in \mathcal C_\Pi$. It has a {centre} line $\ell_\omega$. Then  $\beta(\omega)$ is the intersection point of $\ell_\omega$ with $H$.  Notice that there is no other $\omega'\in \mathcal C_\Pi$ such that $\ell_{\omega}=\ell_{\omega'}$. In fact, if this were the case, $\omega$ and $\omega'$ would have the same {centre} and 
    the same would hold      for all the complexes in the pencil generated by $\omega$ and $\omega'$. This would imply that the family of {centre} lines of the complexes in $\mathcal C_\Pi$ would have dimension 
{smaller than} $\dim (\mathcal C_\Pi)=n-2$, and this is not possible because the family of {centre} lines fills up the hypersurface $P_\Pi$ of $\bbP^n$. This, together with  (ii) of Proposition \ref {prop:pal}    and Lemma \ref{lem:proj},     implies that $\beta$ is also     birational. The assertion immediately follows. \end{proof}

\begin{remark}\label{rem:fan} In the case $k=1$ Theorem \ref {thm:bir} is trivial, since $X_\Pi$ and $\mathcal C_\Pi$ are both irreducible conics. The case $k=2$ is due to Fano (see \cite {Fa}), and proves that a smooth cubic threefold is birational to a prime Fano threefold of genus $8$ and degree 14 in $\bbP^9$. See also \cite{FS} for more details on this case.
\end{remark}

\section{Analysis of the birational map}\label{sec:an}

In \cite {Fa}, Fano makes a detailed analysis of the birational map 
$$\phi:={\beta^{-1}\circ \alpha:  X_\Pi \dasharrow \mathcal C_\Pi}$$
in the case $k=2$.    This has been written in modern terms in \cite [Chapt. III, \S 1]{Isk},  \cite {B1}. Inspired by this,    and     keeping all the above notation, we will now make a similar analysis for all cases $k\geq 2$. Recall that $n=2k+1$.

\subsection{Indeterminacy loci} 

\begin{proposition}\label{prop:indet} 
The indeterminacy locus  of the birational map $\alpha: X_\Pi\dasharrow H$  is a smooth  variety ${\mathcal R_{\Pi,\Sigma}}\subset X_\Pi$ of dimension $n-4$ and degree $$\deg({\mathcal R_{\Pi,\Sigma}})=\frac {(2n-4)!}{(n-1)!(n-2)!},$$
with trivial canonical bundle, that spans a linear space of dimension $\frac {n(n-3)}2-1$. The variety  ${\mathcal R_{\Pi,\Sigma}}$ is regular as soon as its dimension is greater than 1 (i.e., if $n\geq 7$).

The variety ${\mathcal R_{\Pi,\Sigma}}$ parametrizes the family of base lines of $\mathcal L_\Pi$ contained in a general hyperplane section $H$ of the Palatini locus $P_\Pi$, that  fills  up a subvariety $\mathcal Y_{\Pi,H}$ of $H$ of dimension $n-3$ whose degree is 
$$
\deg (\mathcal Y_{\Pi,H})=\frac {n(n-3)}2.
$$
\end{proposition}

\begin{proof} The indeterminacy locus of $\alpha$ consists of  all points in $X_\Pi$ corresponding to base lines of $\mathcal L_\Pi$ that    lie in $H$.     Restricting the line complexes of $\mathcal L_\Pi$ to the hyperplane $\Sigma$ in $\bbP^{n}$ spanned by $H$, one gets a linear system $\mathcal L'_\Pi$ of dimension $n-1$ of line complexes of $\bbP^{n-1}$, whose base locus ${\mathcal R_{\Pi,\Sigma}}$ is the section of $\mathbb G(1, n-1)$ with a general linear space of codimension $n$ in $\bbP^{\frac {n(n-1)}2-1}$ {(see Lemma \ref{lem:projR}, with $c=n$ and $\pi=\Sigma$)},     therefore { its dimension is}     $n-4$ and degree equal to
$$
\deg(\mathbb G(1,n-1))=\frac {(2n-4)!}{(n-1)!(n-2)!}.
$$
By the Adjunction {Formula} ${\mathcal R_{\Pi,\Sigma}}$ has trivial canonical  bundle, since 
$$
K_{\mathbb G(1,n-1)}=\mathcal O_{\mathbb G(1,n-1)}(-n),
$$ 
and it is regular if {its dimension is  at least 2, by the Lefschetz theorem}.    

The subvariety $\mathcal Y_{\Pi,H}$ of the hyperplane $\Sigma$  filled up by the lines in ${\mathcal R_{\Pi,\Sigma}}$ has dimension $n-3$ and its degree is given by the number of these lines that intersect a general plane of $\Sigma$. This is the
same as intersecting a Schubert cycle of type $(2,n-1)$ in $\mathbb G(1,n-1)$ {with }
a general linear space of codimension $n$ in $\bbP^{\frac {n(n-1)}2-1}$, which gives the degree of the aforementioned Schubert cycle. The claim follows from Lemma \ref {lem:0}. \end{proof}

\begin{proposition}\label{prop:2} The indeterminacy locus  of the birational map $\beta: \mathcal C_\Pi\dasharrow H$ consists of the singular locus of $\mathcal C_\Pi$ plus  a  variety $Z_{\Pi,H}$ of dimension $n-4$ and degree 
$$\deg(Z_{\Pi,H})=\frac {k(k+1)(2k+1)}6.$$
The variety $Z_{\Pi,H}$ parametrizes the family of Palatini lines of $\mathcal C_\Pi$ contained in a general hyperplane section $H$ of the Palatini locus $P_\Pi$, that  fills up a subvariety $\mathcal Z_{\Pi,H}$ of dimension $n-3$.\end{proposition}

\begin{proof} The indeterminacy locus  of $\beta$ clearly contains the singular locus of $\mathcal C_\Pi$. Away from ${\rm Sing}(\mathcal C_\Pi)$,  the indeterminacy locus  of $\beta$ consists of all smooth points in $\mathcal C_\Pi$ whose corresponding {centre} lines sit in $H$, i.e., that sit in the hyperplane $\Sigma$ of $\bbP^n$ spanned by $H$. This is just the intersection $Z_{\Pi,H}$ of $V_\Sigma$ (see Lemma \ref {lem:1}) with the linear system $\mathcal L_\Pi$, and, by the generality of $\Pi$ we see that $\dim (Z_{\Pi,H})=n-4$ and $\deg(Z_{\Pi,H})=\deg (V_\Sigma)$. \end{proof}

\begin{remark}\label{rem:pal} Computing the degree of $\mathcal Z_{\Pi,H}$ is in general complicated and  we will not need it. However in the case $k=2$ ({Fano's} case) it can be easily computed to be 10. Indeed in this case $\mathcal Z_{\Pi,H}$ sits in $\Sigma \cong \bbP^4$ and the degree of $\mathcal Z_{\Pi,H}$ is the number of Palatini lines intersecting a general plane of $\Sigma$, or, equivalently, that intersect a general linear subspace $P$ of dimension $3$ of $\bbP^5$.  This is the same as intersecting $\Lambda_P$ (see Lemma \ref {lem:1.2}), {    which is } a quadric, with $Z_{\Pi,H}${, which is} a quintic curve, so that the degree is 10. 

Moreover in this case it is easy to see that $Z_{\Pi,H}$ is a smooth curve of genus 1. Indeed, $Z_{\Pi,H}$ is a general curve section of $V_\Sigma$ (see Lemma \ref {lem:1}), that in turn is a cone in the dual Pl\"ucker space $\bbP^{14}$ with vertex a linear space of dimension 4, over the dual of $\mathbb G(1,4)$, that coincides with $\mathbb G(1,4)$. This is the same as {   considering} a general curve section of $\mathbb G(1,4)$, {   which} is well known to be a smooth elliptic curve.
\end{remark}  

Now the following two {propositions} are obvious.

\begin{proposition}\label{prop:indet2} 
The indeterminacy locus  of the birational map $\phi: X_\Pi\dasharrow \mathcal C_\Pi$ consists of the 
 variety ${\mathcal R_{\Pi,\Sigma}}$ of dimension $n-4$ (see Proposition \ref {prop:indet}) plus the set $Y'_\Pi$ of points $x\in X_\Pi$ not on ${\mathcal R_{\Pi,\Sigma}}$ such that $\alpha(x)\in H$ sits on infinitely many Palatini lines in  $P_\Pi$. 
 \end{proposition}
 
 \begin{proposition}\label{prop:indet3} 
The indeterminacy locus  of the birational map $\phi^{-1}: \mathcal C_\Pi\dasharrow X_\Pi$ consists of ${\rm Sing}(\mathcal C_\Pi)$, of the 
  variety $Z_{\Pi,H}$ of dimension $n-4$ (see Proposition \ref {prop:2}) plus the set $Z'_\Pi$ of smooth points $\omega\in \mathcal C_\Pi$ not on $Z_{\Pi,H}$ such that $\beta(\omega)\in H$ sits on infinitely many base lines in  $P_\Pi$. 
 \end{proposition}

\subsection{The linear systems determining the birational maps}

\begin{proposition}\label{prop:mapalfa} Consider the birational map $\alpha^{-1}: H \dasharrow X_\Pi$. Then the pull--back to $H$ via $\alpha^{-1}$ of the linear system  $|\mathcal O_{X_\Pi}(1)|$, is {contained in} the linear system {$|\mathcal Y_{\Pi,H}+\mathcal O_H(1)|$ and contains $\mathcal Y_{\Pi,H}+ |\mathcal O_H(1)|$ ({so} it coincides with $|\mathcal Y_{\Pi,H}+\mathcal O_H(1)|$, see Remark \ref{rem:compl})}. 
\end{proposition}

\begin{proof} Consider a general linear subspace $P$ of dimension $n-2$ contained in the  linear  span $\Sigma\cong \bbP^{n-1}$ of $H$. The set of lines of $\mathbb P^n$ intersecting $P$ is a line complex $\omega_P$, hence it gives a hyperplane section $\Omega_P$ of $X_\Pi$. The pull back of $\Omega_P$ via $\alpha^{-1}$ clearly consists of $\mathcal Y_{\Pi,H}$ plus the hyperplane section of $H$ with $P$. {Since all hyperplane sections of $X_\Pi$ are linearly equivalent,} the statement follows. 
\end{proof}

As an immediate consequence of Proposition \ref {prop:mapalfa} we have:

\begin{corollary}\label{prop:proj} $H$ is the projection of $X_\Pi$ from the linear space spanned by ${\mathcal R_{\Pi,\Sigma}}$. {In  this projection the total transform of ${\mathcal R_{\Pi,\Sigma}}$ is $\mathcal Y_{\Pi,H}$. } 
\end{corollary}

\begin{proposition}\label{prop:mapbeta} Consider the birational map $\beta: \mathcal C_\Pi\dasharrow H$. Then the pull--back to $\mathcal C_\Pi$ via $\beta$ of the linear system $|\mathcal O_H(1)|$, is the linear subsystem $	\mathcal S_{\Pi,H}$ of $|\mathcal O_{\mathcal C_\Pi}(k)|$ {cut out on $\mathcal C_\Pi$ by the hypersurfaces $\Lambda_P$ of Lemma \ref{lem:1.2}, with $P\subset \Sigma$ {a linear space of dimension $n-2$}, where $\Sigma\cong\bbP^{n-1}$ is the hyperplane spanned by $H$, i.e., by the $n$ Pfaffians of degree $k$ of the restrictions to $\Sigma$ of the complexes of $\mathcal L_\Pi$. Its base locus is the indeterminacy locus of $\beta$}.
\end{proposition}

\begin{proof} This is an immediate consequence of Lemmata \ref {lem:1.2} and \ref {lem:2}. \end{proof}

\begin{remark}\label{rem:eur} In this remark we make {a heuristic} argument that suggests that the Palatini locus is singular in dimension $n-4$ and therefore its general hyperplane section $H$ is singular in dimension $n-5$ (this is by the way what happens for $k=2$ as shown in \cite {Fa, Isk}). 

In the indeterminacy locus of $\beta$, we have the variety $Z_{\Pi,H}$ of dimension $n-4$ sitting in $\mathcal C_\Pi\subset \mathcal L_\Pi\cong \bbP^{n-1}$. Since $n=2k+1$, we expect a family of dimension $n-3$ of $k$--secant lines to $Z_{\Pi,H}$, that fills  up a variety $\bar Z_{\Pi,H}$ of dimension $n-2$.  Two conditions are required, for the lines in $\bar Z_{\Pi,H}$, to be contained in $\mathcal C_\Pi$, so we expect a family of dimension $n-5$ of $k$-secant lines to $Z_{\Pi,H}$ contained in $\mathcal C_\Pi$, forming a variety $Z'_\Pi$ of dimension $n-4$ ruled by lines.  Proposition \ref {prop:mapbeta}  yields that $Z'_\Pi$ is contracted by $\beta$ to a variety of dimension $n-5$ of singular points of $H$, as announced. 

Note that the divisors in $\mathcal S_{\Pi,H}$ are complete intersections of two hypersurfaces of degree $k$ and $k+1$ in $\bbP^{2k}$, so their canonical sheaf is trivial. The general hyperplane sections of $H$ are hypersurfaces of degree $2k$ in $\bbP^{2k-1}$, so that their dualizing sheaf is also trivial. We expect that  the singularities of $P_\Pi$ and those of $H$ do not impose conditions to adjunction. This is in fact what happens for $k=2$, in which case they are double points (exactly 25 double points for $H$, see \cite {Fa, Isk}). 
\end{remark}

\begin{proposition}\label{prop:mapbetainv} Consider the birational map $\beta^{-1}: H \dasharrow \mathcal C_\Pi$. Then the pull--back to $H$ via $\beta^{-1}$ of the linear system $|\mathcal O_{\mathcal C_\Pi}(1)|$, is 
a linear system $|M_{\Pi,H}|$ whose general member is a divisor on $H$ of degree 
$$
\frac {n(n-3)}2+2.
$$
\end{proposition}

\begin{proof} This is an immediate consequence of Lemma \ref {lem:3}. \end{proof}

At this point we recall the following:

\begin{theorem}\label{conj:main} Let $k\geq 2$ be an integer. Consider the section $Z$ of $\check {\mathbb G}(1,2k)$ with a general linear space ${\Theta}$ of dimension $2k$ in $\check \bbP^{2k^2+k-1}$, so that $Z$ is irreducible of dimension $2k-3$. Consider the variety $Z_k$ filled up by the $k$--secant lines to $Z$. Then $Z_k$ is a hypersurface of ${\Theta}\cong \bbP^{2k}$ of degree $2k^2-k-1$ with multiplicity $2k-1$ along $Z$. 
\end{theorem}

\begin{proof} As indicated to us by F. Russo, this basically follows from \cite [\S 4.2]{ESB} (see also \cite [\S 5]{HKS}). We sketch the proof for the reader's convenience.

The ideal of the dual variety of $\mathbb G(1,2k)$ ({whose singular locus has codimension ten in $\check \bbP^{2k^2+k-1}$}) is generated by $2k+1$ forms of degree $k$ that define
a dominant rational map  ${\check \bbP^{2k^2+k-1}}\dasharrow  \bbP^{2k}$ that maps a general line complex to its unique point of indeterminacy.  This rational map has linear fibres of dimension complementary to $2k$,
and therefore, restricting it to a general linear space ${\Theta}$ of dimension $2k$ yields a variety $Z$ of dimension $2k-3$ (the intersection with $\check {\mathbb G}(1,2k)$), which is the indeterminacy locus of a Cremona transformation given by forms of degree $k$ (and which is smooth for $k\leq 4$).
The inverse map is given by forms of degree $2k-1$ (see \cite [Lemma 2.4]{ESB}), and the degree of $Z_k$ is $2k^2-k-1$ and has points of multiplicity $2k-1$ along $Z$. In fact if  $E_1$ and $E_2$ are the exceptional divisors of the blow--ups of the base loci of the above Cremona transformation, we have
$$H_2=k H_1 - E_1,\,\,  H_1=(2k-1)H_2-E_2$$
where $H_1$ is the hyperplane class of ${\Theta}$ and $H_2$ is the hyperplane class of $\bbP^{2k}$.
Then
$$E_2=(2k-1)H_2-H_1=(2k^2-k-1)H_1-(2k-1)E_1$$
as wanted. \end{proof}

\begin{remark}\label{rem:conj} 
Concerning Theorem \ref {conj:main}, there are various interesting questions that are in order. For example:\\
\begin{inparaenum}[(i)]
\item give information on the linear sections of $\check {\mathbb G}(1,2k)$: their canonical bundle, their arithmetic genus,  for the surfaces the value of $K^2$,  etc.;\\
\item given an irreducible, projective variety $V$ of dimension $m$ in $\bbP^n$, the family of $k$ secant lines to $V$ has {expected} dimension $km+k-2-(k-2)n$, and so we expect them to fill up a variety $V_k$ of dimension $\min \{n,km+k-1-(k-2)n\}$. When does it fail to do so? Suppose the above minimum is $km+k-1-(k-2)n$, and that the variety $V_k$ has this dimension. Then what is the degree of $V_k$, and what is its multiplicity along $V$? It is also intriguing the case where $km+k-1-(k-2)n=n$ and $V_k= \bbP^n$, and only one $k$--secant line to $V$ passes through the general point of $\bbP^n$. Characterize such varieties $V$. This is an extension of  the concept of  OADP variety.
\end{inparaenum} 
\end{remark}

\begin{theorem}\label{thm:m1}  Consider the birational map $\phi^{-1}: \mathcal C_\Pi\dasharrow X_\Pi$. Then the pull--back via $(\phi^{-1})^*$ of the linear system $|\mathcal O_{X_\Pi}(1)|$ is {a} linear subsystem of $|\mathcal O_{\mathcal C_\Pi}(2k^2-1)|$ {whose} divisors {pass} through $Z_{\Pi,H}$ with multiplicity $2k$.
\end{theorem}

\begin{proof} {Let $M\in |M_{\Pi,H}|$ be general. By Proposition \ref{prop:mapbetainv}, $M$ is the pull--back via $\beta^{-1}$ of the section of $\mathcal C_\Pi$ with a general hyperplane $\mathcal L'$ of $\mathcal L_\Pi$, i.e., $M$ is the set of points of $H$ lying in the centre of some complex of $\mathcal L'$, so that $M=V_{\mathcal L'}\cap \Sigma$, where $\Sigma\cong \bbP^{n-1}$ is the span of $H$ (see \S \ref{ssec:cong} and Lemma \ref{lem:3}). Let $\mathcal L'_\Sigma$ be the linear system of complexes of $\Sigma$ obtained by restricting to $\Sigma$ the complexes of $\mathcal L'$. Since $\mathcal L'$ is a general linear system of dimension $n-2$, $\mathcal L'_\Sigma$ is a general linear system of dimension $n-2$ of complexes of $\Sigma\cong \bbP^{n-1}$, with $n-1$ even, i.e., we are in the setting of Section \ref{sec:even}, with $n-1$ in place of $n$. By Proposition \ref{prop:pal2} {below}, the Palatini locus $F:=P_{\mathcal L'_\Sigma}$ of $\mathcal L'_\Sigma$ is a hypersurface of degree $n-2$ of $\Sigma$, which coincides with the locus of base lines of $\mathcal L'_\Sigma$. The hypersurface $F$ contains $M$, because a point of $\Sigma$ lying in the centre of a complex $\omega\in \mathcal L'$ lies in the centre of the restriction of $\omega$ to $\Sigma$. Moreover $F$ contains $\mathcal Y_{\Pi,H}$, because the lines that sweep out $\mathcal Y_{\Pi,H}$ are base lines of $\mathcal L_\Pi$ contained in $\Sigma$, hence they are base lines of $\mathcal L'_\Sigma$. (Note that, by Proposition \ref{prop:map}, see \eqref{eq:exp}, $F$ is the unique hypersurface of degree $n-2$ of $\Sigma$ containing $M$.)} If $R$ is the residual intersection of $F$ with $H$ off $M$ we have
$$\deg(R)=(n-2)(n-1)-\frac {n(n-3)}2-2=\frac {n(n-3)}2.$$
{By the above, $R$ contains $\mathcal Y_{\Pi,H}$, and since they have the same degree, they coincide.} In conclusion we have
\begin{equation}\label{eq:kil}
|\mathcal Y_{\Pi,H}+M_{\Pi,H}| \sim |\mathcal O_H(n-2)|.
\end{equation}
Now, taking into account Propositions \ref {prop:mapalfa}, \ref {prop:mapbeta} and \ref {prop:mapbetainv}, we conclude the proof. \end{proof}

{Note that the above proof does not use Theorem \ref{conj:main}, which is however consistent with it. Indeed, by Lemma \ref{lem:1}, projecting from the vertex of $V_\Sigma$, the variety $Z_{\Pi,H}=V_\Sigma\cap \mathcal L_\Pi$ is isomorphic to the section of $\check {\mathbb G}(1,2k)$ with a general linear space of dimension $2k$. Hence, by Theorem \ref{conj:main}, the variety $Z_k$ filled up by the $k$--secant lines to $Z_{\Pi,H}$ is a hypersurface of $\mathcal L_\Pi\cong\bbP^{2k}$ of degree $2k^2-k-1=(n-2)k-1$, with multiplicity $2k-1=n-2$ along $Z_{\Pi,H}$, so that $Z_k$ plus a hyperplane of $\mathcal L_\Pi$ has the same degree and the same multiplicity along $Z_{\Pi,H}$ as the pull--back via $\beta$ of a hypersurface of degree $n-2$ of $\Sigma$ (see Proposition \ref{prop:mapbeta}).}

\begin{remark}\label{rem:compl} {The pull--back to $H$ via $\alpha^{-1}$ of $|\mathcal O_{X_\Pi}(1)|$ is the complete linear system cut out on $H$, off a general $M\in |M_{\Pi,H}|$, by the hypersurfaces of degree $n-1$ of $\Sigma$ containing $M$, i.e., it is $|\mathcal Y_{\Pi,H}+\mathcal O_H(1)|$. Indeed, by Proposition \ref{prop:mapalfa} and \eqref{eq:kil}, if $D$ is the pull--back of a hyperplane section of $X_\Pi$, then $D+M\sim \mathcal Y_{\Pi,H}+M+\mathcal O_H(1)\sim \mathcal O_H(n-1)$, so that $D+M$ is cut out on $H$ by a hypersurface of degree $n-1$ of $\Sigma$ containing $M$. As we saw in the proof of Theorem \ref{thm:m1}, $M=V_{\mathcal L'}\cap \Sigma$, with $\mathcal L'$ a general linear system of complexes of dimension $n-2$, hence, by Proposition \ref{prop:map} (see \eqref{eq:exp}), $h^0(\Sigma,\mathcal I_{M|\Sigma}(n-1))={n\choose 2}+1$. Since $H$ is one of these hypersurfaces, the linear system in question has dimension ${n\choose 2}-1=\dim(\Pi)$, which is also the dimension of the pull--back of $|\mathcal O_{X_\Pi}(1)|$, because $X_\Pi$ is non--degenerate in $\Pi$. Hence the two linear systems coincide. This is a further instance of the interplay between the three parts of the paper.}
\end{remark}

\begin{theorem}\label{thm:m2}  Consider the birational map $\phi: X_\Pi \dasharrow \mathcal C_\Pi$. Then the pull--back via $\phi^*$ of the linear system $|\mathcal O_{\mathcal C_\Pi}(1)|$ is {a} linear subsystem of $|\mathcal O_{X_\Pi}(n-2)|$ {whose} divisors {pass} through ${\mathcal R_{\Pi,\Sigma}}$ with multiplicity $n-1$.
\end{theorem}

\begin{proof} This is an immediate consequence of Proposition \ref {prop:mapbetainv}, Corollary \ref {prop:proj}
and \eqref {eq:kil}. 
\end{proof}

\section{Lines}\label{sec:lines}

In this section we study some properties of the lines lying on $X_\Pi$ and on $\mathcal C_\Pi$.

\subsection{Lines on $X_\Pi$} 

\begin{proposition}\label{prop:xpi} In the setup of Section \ref {sec:fano}, there is an $(n-4)$--purely dimensional family $\mathcal F_{X_\Pi}$ of lines on $X_\Pi$.
\end{proposition} 

\begin{proof} The lines on $\mathbb G(1,2k+1)$ form an irreducible family of dimension $6k-1$ ({which is} isomorphic to the universal bundle over $\mathbb G(2,2k+1)$). Since $X_\Pi$ is cut out on $\mathbb G(1,2k+1)$ by the intersection of $n=2k+1$ general hyperplanes, we see that the family of lines on $X_\Pi$ has dimension
$$
6k-1-2(2k+1)=2k-3=n-4
$$
as wanted.
\end{proof}

The union of the lines of $\mathcal F_{X_\Pi}$ is a divisor $F_\Pi$ on $X_\Pi$. Since ${\rm Pic}(X_\Pi)=\mathbb Z\langle \mathcal O_{X_\Pi}(1)\rangle$, there is a positive integer $m_k$ such that $F_\Pi\in |\mathcal O_{X_\Pi}(m_k)|$. For $k=2$ it is well known that $m_2=5$ (see \cite {Fa, Isk}) and it is a problem to compute $m_k$ for $k>2$. However we will not dwell on this here. 

Notice that a line on $X_\Pi$ gives rise to a pencil of base lines contained in the Palatini locus $P_\Pi=B_\Pi$ (recall Proposition \ref {prop:pal}). Hence there is {an $(n-4)$--dimensional} subvariety ${S_{\Pi}}\subset P_\Pi$ of points $p$ such that through $p$ pass a 1--dimensional family of lines forming a pencil (thereby {sweeping} out a plane). 

\begin{lemma} \label{lem:lop} The Palatini locus $P_\Pi$ is smooth at any point $x$ such that there is a unique base line passing through $x$. \end{lemma}

\begin{proof} {Let $x\in B_\Pi=P_\Pi$ be a point that lies on a unique base line $\ell$. Consider the tautological morphism $f:\mathcal X_\Pi\to P_\Pi=B_\Pi$, where $\mathcal X_\Pi\to X_\Pi$ is the $\bbP^1$--bundle of the base lines; $f$ is proper and, by Proposition \ref{prop:pal}, (ii), it is birational. The scheme--theoretic fibre $f^{-1}(x)$ is isomorphic to $\mathbb P^{n-1}_x\cap \Pi$, where $\mathbb P^{n-1}_x\subset \mathbb G(1,n)$ is the linear space of the lines through $x$ (see Section \ref{sec:lc}). Being the intersection of two linear spaces, $f^{-1}(x)$ is reduced, and since it is supported at the point $[\ell]$, it is a reduced point. By the upper semicontinuity of the dimension of the fibres, $f$ is quasi--finite, hence finite {over a}
neighbourhood of $x$, and by Nakayama's lemma $f$ is a closed embedding over a (possibly smaller) neighbourhood of $x$. Since $f$ is birational and $P_\Pi$ is irreducible, $f$ is then an isomorphism over a neighbourhood of $x$. Since $X_\Pi$, and hence $\mathcal X_\Pi$, is smooth, $P_\Pi$ is smooth at $x$.} \end{proof}

This lemma and the above considerations show that 
$$
\dim ({\rm Sing}(P_\Pi))\leq n-4 
$$
and therefore
\begin{equation}\label{eq:wor}
\dim ({\rm Sing}(H))\leq n-5.
\end{equation}

We can actually show that:

\begin{proposition}\label{prop:h} With the notation of Section \ref {sec:fano}, we have that the singular locus of $H$ has dimension $n-5$, equivalently the singular locus of the Palatini locus has  dimension $n-4$. 
\end{proposition}

\begin{proof} When $H$ moves, ${\mathcal R_{\Pi,\Sigma}}$ also moves filling up the whole of $X_\Pi$. Therefore the divisor  $F_\Pi$ described by all lines in $X_\Pi$ does not contain ${\mathcal R_{\Pi,\Sigma}}$ for general $H$. Hence $F_\Pi$ intersects ${\mathcal R_{\Pi,\Sigma}}$ in codimension 1, i.e., in dimension $n-5$, hence there is a family of dimension $n-5$ of lines $\ell$ in $X_\Pi$ that intersect ${\mathcal R_{\Pi,\Sigma}}$. Such a line $\ell$ is contracted to a point of $H$ by the projection $\alpha: X_\Pi\dasharrow H$ from the span of ${\mathcal R_{\Pi,\Sigma}}$, and therefore $\alpha$ is a small contraction, yielding that $H$ is singular in dimension $n-5$, as wanted.\end{proof}

{Moreover, we have}:

\begin{corollary}\label{cor:lop} The Palatini locus $P_\Pi$ is singular exactly at the points $p$ such that there are infinitely many base lines passing through $p$.
\end{corollary}

\begin{proof} {Let $p\in P_\Pi=B_\Pi$. The base lines through $p$ correspond to the points of the linear space $\mathbb P^{n-1}_p\cap \Pi$, so that either there is a unique base line through $p$, and then $P_\Pi$ is smooth at $p$ by Lemma \ref{lem:lop}, or there are infinitely many. In the latter case the fibre over $p$ of the tautological morphism $f:\mathcal X_\Pi\to P_\Pi$, which is birational by Proposition \ref{prop:pal}, (ii), has positive dimension. On the other hand, the exceptional locus of $f$, i.e., the set of points of $\mathcal X_\Pi$ lying in positive dimensional fibres of $f$, consists of pairs $([\ell],x)$ such that $[\ell]$ lies on a line of $X_\Pi$ corresponding to a pencil of lines with centre $x$, hence, by Proposition \ref{prop:xpi}, it has dimension at most $n-3$, i.e., codimension at least 2 in $\mathcal X_\Pi$. Since $\mathcal X_\Pi$ is smooth, by  {van der Waerden's purity theorem} (the exceptional locus of a birational morphism onto a smooth variety has pure codimension 1, see \cite[Chap. II, \S 4.4]{Sha}), $P_\Pi$ cannot be smooth at $p$.} \end{proof} 

The following is also easy:

\begin{corollary}\label{cor:linH} In the projection $\alpha: X_\Pi\dasharrow H$ from the span of ${\mathcal R_{\Pi,\Sigma}}$ the image of a general line of $X_\Pi$ is a line on $H$ {which is} $(n-2)$--secant {to} the divisors in $|M_{\Pi,H}|$.
\end{corollary}

\begin{proof} By Proposition \ref {prop:mapalfa}, the pull--back to $H$ via $\alpha^{-1}$ of the linear system $|\mathcal O_{X_\Pi}(1)|$, is the linear system $|\mathcal Y_{\Pi,H}+\mathcal O_H(1)|$. On the other hand we have \eqref {eq:kil}, hence 
$$
|\mathcal Y_{\Pi,H}+\mathcal O_H(1)|=|\mathcal O_H(n-1)-M_{\Pi,H}|
$$
and the assertion follows.
\end{proof}

\subsection{Lines on $\mathcal C_\Pi$}

\begin{proposition}\label{prop:linpfaff}  In the setup of Section \ref {sec:fano} (hence $n=2k+1$), there is a $(3k-4)$--purely dimensional family  of lines on $\mathcal C_\Pi$.
\end{proposition} 

\begin{proof} One knows that the universal Pfaffian  $\check {\mathbb G}(1,2k+1)$ contains a family of lines of dimension $4k^2+5k-4$ (this follows from \cite[Cor. 3] {B}). Then $\mathcal C_\Pi$ is a general section of $\check {\mathbb G}(1,2k+1)$ with a linear space of codimension $2k^2+k$, hence $\mathcal C_\Pi$ contains a family of lines of dimension
$$
4k^2+5k-4-2(2k^2+k)=3k-4.
$$
\end{proof}

\begin{remark}\label{rem:lines}
(i) For $k=2$ the Fano surface of a smooth cubic hypersurface in $\bbP^4$ is a canonical surface embedded in $\bbP^9$ with the Pl\"ucker embedding, having $q=5$ and $K^2=45$. It would be nice to understand the characters of the Fano variety of lines of $\mathcal C_\Pi$ in general. \smallskip

(ii) Note that the general line $\ell$ in $\mathcal C_\Pi$ cannot intersect $Z_{\Pi,H}$. Indeed, $\ell$ corresponds to a pencil of lines in a plane $P$, with {centre} a point $p\in P$. On the other hand $Z_{\Pi,H}$ is the set of Palatini lines contained in a general hyperplane ${\Sigma}$ of $\mathbb P^n$. So, if ${\Sigma}$ does not contain $p$ the line $\ell$ cannot intersect $Z_{\Pi,H}$. As a consequence, by Theorems \ref {thm:m1} and \ref {thm:m2}, the lines on $\mathcal C_\Pi$ pull back via $\phi: X_\Pi \dasharrow \mathcal C_\Pi$ to rational curves of degree  $2k^2-1$ that are $(2k^2-k-1)$--secant to ${\mathcal R_{\Pi,\Sigma}}$. \smallskip

\end{remark}

\section{The even case}\label{sec:even}

\subsection{Preliminaries}  In this section we will consider $\mathbb G(1,n)$ with $n=2k\geq 4$,  and  we will fix again a (general) linear space $\Pi$ of  $\bbP^{\frac {n(n+1)}2-1}$ of codimension $n$ hence of dimension
$$
\dim (\Pi)=\frac {n(n-1)}2-1,
$$
that intersects $\mathbb G(1,n)$ in a smooth, irreducible variety $X_\Pi$ of codimension $n$, hence of dimension $n-2=2k-2$ and degree
$$
\deg (X_\Pi)=\deg(\mathbb G(1,n))=\frac {(2n-2)!}{n!(n-1)!}.
$$
By the Adjunction Formula one has
$$
K_{X_\Pi}\sim \mathcal O_{X_\Pi}(-1),
$$
so $X_\Pi$ is a Fano variety of index 1. 

Related to $\Pi$ and $X_\Pi$,    the following objects {have been} introduced in \S  \ref {sserc:ls}:
\\
\begin{inparaenum}
\item [$\bullet$] the linear system $\mathcal L_\Pi$, that has dimension $n-1=2k-1$;\\
\item [$\bullet$] the locus of base lines $B_\Pi$ in $\bbP^n$, with expected dimension $n-1=2k-1$;\\
\item [$\bullet$] the Palatini locus $P_\Pi$ in $\bbP^n$, also with expected dimension $n-1=2k-1$.
\end{inparaenum}

\begin{proposition}\label{prop:pal2} In the above setup, if $\Pi$ is sufficiently general, one has:\\
\begin{inparaenum}[(i)]
\item the Palatini locus $P_\Pi$ is a hypersurface of degree $n-1=2k-1$ in $\mathbb P^n$;\\
\item  the general point $p\in P_\Pi$ is the {centre} of a unique complex in $\mathcal L_\Pi$ and therefore $P_\Pi$ is rational;\\
\item the Palatini locus $P_\Pi$ coincides with the locus of base lines $B_\Pi$ and therefore $X_\Pi$ is unirational;\\
\item through the general point of the Palatini locus $P_\Pi$ passes a unique base line.\end{inparaenum}
\end{proposition}

\begin{proof} The proof is similar to the one of Proposition \ref {prop:pal}. 

   Take a general line $r$ in $\bbP^n$. As in the proof of Proposition \ref{prop:pal}, $r\cap P_\Pi$ is the set of points of $r$ lying in the {centre} of some complex of the {   $(n-2)$-dimensional } linear system of the complexes of $\mathcal L_\Pi$ containing $r$.    Hence, by Lemma \ref{lem:line}, $r$ intersects $P_\Pi$ in $n-1=2k-1$ points, which  proves  (i).

Next we claim that a general point $p$ of $P_\Pi$ is the    {centre}    of a unique complex in $\mathcal L_\Pi$. In fact, suppose to the contrary that $p$ is the {centre} of $\omega$ and ${\omega'}$, with $\omega, \omega'$ two distinct line complexes of $\mathcal L_\Pi$. Then $p$, being singular   for both $\omega$ and $ \omega'$, would be singular for all the complexes of the pencil generated by $\omega$ and $ \omega'$. This is impossible because $P_\Pi$ is a hypersurface. Then (ii) follows, because the map $\omega\in \mathcal L_\Pi\dasharrow p_\omega\in P_\Pi$ is birational.    

Let $p$ be a general point of $P_\Pi$.    By (ii) and Lemma \ref{lem:hyp}, there is a unique base line $b$ of $\mathcal L_\Pi$ passing through $p$.     So the general point $p$ of $P_\Pi$ sits in $B_\Pi$. Since {$B_\Pi$ is irreducible of dimension at most $\dim(X_\Pi)+1=n-1$ and contains $P_\Pi$, which has dimension $n-1$}, this proves (iii) and (iv). \end{proof}

    As in the odd case, $P_\Pi$ is the degeneracy locus of a general morphism $\mathcal O_{\bbP^n}^{\oplus n}\to \Omega_{\bbP^n}(2)$, so part (i) of Proposition \ref{prop:pal2} and the parametrization $\omega\in \mathcal L_\Pi \dasharrow p_\omega\in P_\Pi$ can also be found in \cite[\S 2]{BM} and \cite[\S 1.2]{M} {(see Remark \ref{rem:lit})}; for $n=4$ one gets the Segre cubic (cf. Remark {\ref{rem:segre}}).    

\begin{remark}\label{rem:lop} In the setup of the proof of Proposition \ref {prop:pal2}, the map $\omega\in \mathcal L_\Pi\dasharrow p_\omega\in P_\Pi$ is not a morphism as soon as $k\geq {2}$. Indeed, the set of singular complexes in $\mathbb P^{2k}$ has dimension $2k^2+k-4$ so that it intersects $\mathcal L_\Pi$ in dimension ${2k-4}$. 
\end{remark}

\subsection{The birationality theorem} Since $P_\Pi$ is rational (see Proposition \ref {prop:pal2}, (ii)), there  exists    a   birational map    
$$
\tau_\Pi: \mathbb P^{n-1}\dasharrow P_\Pi\subset \mathbb P^n.
$$
The next theorem determines such a map.    First,     we need some notation. Let $\omega_1,\ldots, \omega_n$ be linearly independent complexes generating $\mathcal L_\Pi$. Suppose that for each $h\in \{1,\ldots, n\}$, $\omega_h$ has equation
\begin{equation*}\label{eq:comp}
{\bf u}= {\bf A}^h\cdot {\bf x}
\end{equation*}
where ${\bf A}^h=(a^h_{ij})_{0\leq i,j\leq n}$ is an antisymmetric matrix. Any  complex in $\mathcal L_\Pi$     then  has an     equation of the form
$$
{\bf u}= \sum_{h=1}^n u_h {\bf A}^h\cdot {\bf x}
$$
where  $[u_1,\ldots, u_n]$ are  projective coordinates in $\mathcal L_\Pi$. 
The matrix
$$
{\bf {\mathcal A}}=\sum_{h=1}^n u_h {\bf A}^h
$$
is an antisymmetric matrix of linear forms in the variables $u_1,\ldots, u_n$, of order $n+1=2k+1$. Let ${\bf {\mathcal A}}_0, \ldots, {\bf {\mathcal A}}_n$ be the Pfaffians of the maximal diagonal minors of ${\bf {\mathcal A}}$,    which are     polynomials of degree $k$ in $u_1,\ldots, u_n$. 

\begin{theorem}\label{thm:det} Up to {projective} transformations, we have
$$
\tau_\Pi: [u_1,\ldots, u_n]\in \mathbb P^{n-1}\cong \mathcal L_\Pi\dasharrow [{\bf {\mathcal A}}_0, \ldots, {\bf {\mathcal A}}_n]\in P_\Pi\subset \mathbb P^n.
$$
\end{theorem} 

\begin{proof} The proof is trivial: the {centre} of the general complex 
$$
{\bf u}= {\bf {\mathcal A}}\cdot {\bf x}
$$
is exactly given by  $[{\bf {\mathcal A}}_0, \ldots, {\bf {\mathcal A}}_n]$. \end{proof}

    An analogous description is given in \cite[\S 2.3]{T1} and \cite[\S 3.2]{T2}, but in    a different case {($m\leq n-1$, see Remark \ref{rem:lit})}. In \cite{T1, T2}, the degeneracy locus is shown to be birational to the blow--up of the parameter space of the linear system along the subscheme cut out by the Pfaffians of the maximal diagonal minors, which for a general system has codimension 3, and is Gorenstein by \cite{BE} (cf. Corollary \ref{cor:bir2} below).    

A few consequences are in order. 

\begin{corollary}\label{cor:bir1} The variety $X_\Pi$ is birational to a general linear  Pfaffian hypersurface of degree $k$ in $\mathbb P^{2k-1}$.
\end{corollary}

\begin{proof}
By Proposition \ref {prop:pal2}, (iii) and (iv),    and Lemma \ref{lem:proj},    $X_\Pi$ is birational to a general linear section of $P_\Pi$, and the claim follows from Theorem 	\ref {thm:det}. 
\end{proof}

Combining Corollary \ref{cor:bir1} with the rationality theorem of Part \ref{part:CY}, one sees that $X_\Pi$ is actually rational (see Corollary \ref{cor:evenrat}).

\begin{corollary}\label{cor:bir2} The base locus $Y_\Pi$ of the linear system of hypersurfaces of degree $k$ in $\mathbb P^{2k-1}$ generated by ${\bf {\mathcal A}}_0, \ldots, {\bf {\mathcal A}}_n$ has codimension 3, it has degree
$$
\frac 1{12}  {{2k+1}\choose 2} \cdot (2k+2)
$$
and,    as soon     as $k\geq 3$, it is an  irreducible canonical variety. 

\end{corollary}

\begin{proof} The degree formula has been proved in \cite [Prop. 12, (c)] {HT}. As for {the rest} of the assertion, note that, by the   well known  structure theorem in \cite {BE}, the variety $Y_\Pi$ defined by the equations 
$$
{\bf {\mathcal A}}_0= \cdots ={\bf {\mathcal A}}_n=0
$$
has codimension $3$ and it is Gorenstein, so it is subcanonical, i.e., there is an integer $e$ such that the dualizing sheaf $\omega_Y$ is isomorphic to $\mathcal O_Y(e)$. Putting together   \cite [Prop. 4.3] {DG} and the remark in \cite {KM} on p. 75 soon after formula (*), it follows that $e=1$, so $Y_\Pi$ is canonical.
\end{proof} 

\begin{remark}\label{rem:segre} For $k=2$, $P_\Pi$ is the famous Segre cubic (the unique cubic hypersurface with 10 nodes) and $Y_\Pi$ is a general set of 5 points in $\mathbb P^3$,    which is     base locus of the linear system of quadrics that map $\mathbb P^3$ birationally onto the Segre cubic. This is a classical result by C. Segre {(see also \cite[Ex. 10.2.20]{Dol})}. 

In general, $P_\Pi$ is an analogue of the Segre cubic. For $k=3$ it is a hypersurface of degree 5 in $\mathbb P^6$,    which is     the image of $\mathbb P^5$ via the linear system of cubics that contain a canonical surface of degree 14 defined by the  Pfaffians of a general antisymmetric matrix of linear forms of    order 7.    

   In general,     the $(k-1)$--secant lines to $Y_\Pi$ are mapped to lines in $P_\Pi$. The family of 
$(k-1)$--secant lines to $Y$ has    the expected     dimension $2k-2$ hence $P_\Pi$ should be filled up by a $(2k-2)$--dimensional family of lines such that    finitely many lines pass through a general point.     

The $k$--secant lines to $Y_\Pi$ are contracted to singular points of $P_\Pi$. The family of 
$k$--secant lines to $Y_\Pi$ has expected dimension $2k-4$     hence,     we expect that $P_\Pi$ is singular in dimension $2k-4$. This fits with    the following      computation. The variety $X_\Pi$ has a family of lines of dimension $2k-4$ and each of these lines contributes to one singular point of $P_\Pi$. 

   Finally,     we remark that the base locus of the linear system generated by $ {\bf {\mathcal A}}_0, \ldots, {\bf {\mathcal A}}_n$ is singular {along a locus of codimension 10 in $\mathbb P^{2k-1}$}, so it is smooth only for $k\leq 5$ (according to Hartshorne's conjecture). 
\end{remark}

   \part{Calabi--Yau linear sections {of Grassmannians of lines}}\label{part:CY}   

\section{The Palatini family}\label{sec:palat}

In what    follows,     we will assume $c=n+1$ and, as usual $n=2k+1$, with $k\geq 2$.    Hence,     $X_\Pi$ will be a Calabi--Yau variety, of dimension $n-3$. The hypersurface $\mathcal C_\Pi$    has dimension      $c-2=n-1$ and has degree $k+1$. Notice that $\mathcal C_\Pi$ is smooth if and only if $n=5$, in which case it is a general Pfaffian cubic hypersurface in $\mathbb P^5$. 

In this    section,     we will study the Palatini family in the Calabi--Yau case $c=n+1$,    adopting      the generality hypothesis for the subspace $\Pi$ of codimension $n+1$ in  $\bbP^{\frac {n(n+1)}2-1}$. Since  $\mathcal C_\Pi$ has  dimension $n-1$, the expected dimension of the Palatini locus $P_\Pi$ is $n$, i.e., we expect that $P_\Pi=\bbP^n$. 

\subsection{The Palatini locus} \label{ssec:palat}

In this section we prove the following basic    result    :

\begin{theorem}\label{thm:basic} In the Calabi--Yau case $c=n+1$ the Palatini locus $P_\Pi$ coincides with $\bbP^n$ and    a unique Palatini line passes through a general point of $\bbP^n$.   
\end{theorem}

\begin{proof} The linear system $\mathcal L_\Pi$ of line complexes can be represented  by  an equation of the type \eqref {eq:compo}, where ${\bf A}$ is an antisymmetric matrix ${\bf A}=(a_{ij})_{0\leq i,j\leq n}$ depending linearly on $n+1$  homogeneous parameters. Fix a general point $p\in \bbP^n$. By choosing suitable coordinates, we may assume that $p$ has homogeneous coordinates $[1,0,\ldots, 0]$. Then $p$ is in the {centre} of a line complex defined by \eqref {eq:compo} if and only if 
\begin{equation}\label{eq:lop}
a_{01}= \cdots =a_{0,n}=0.
\end{equation}
These are $n$ linear equations in the $n+1$ parameters, so the system has    a non--trivial     solution, and therefore $p$ is in the {centre} of some singular line complex $\omega$ in $\mathcal L_\Pi$. By the generality of $p$, $\omega$ is a smooth point of $\mathcal C_\Pi$. This proves that $P_\Pi$ coincides with $\bbP^n$. The same argument 
   also proves      that  a unique Palatini line passes  through a  general point of $\bbP^n$. Indeed if the system \eqref {eq:lop} has more than one solution, it has infinitely many solutions. But, since $P_\Pi=\bbP^n$,  it is not possible that     infinitely many Palatini lines pass through a general point of $\bbP^n$.     This proves the assertion.
\end{proof}

\subsection{The rationality theorem}\label{ssec:rat}

As an immediate consequence of Theorem \ref {thm:basic}, we have the following:

\begin{theorem}[The rationality theorem]\label{thm:rat} Let $n=2k+1$ and let $m\geq n$ be an integer. Then the general Pfaffian hypersurface of degree $k+1$ in $\bbP^m$, i.e., the hypersurface defined by the Pfaffian of a general antisymmetric matrix of order $n+1$ of linear forms in $m+1$ variables, is rational.
\end{theorem}

\begin{proof} We distinguish three cases.

\emph{Case $m=n$.} The general Pfaffian hypersurface of degree $k+1$ in $\bbP^n$ is of the form $\mathcal C_\Pi$,    where $\Pi$ is a general subspace     of codimension $n+1$ in $\bbP^{\frac {n(n+1)}2-1}$. Recall that we have the Palatini family $\phi_\Pi: \mathcal P_\Pi\longrightarrow U_\Pi$, where $U_\Pi$ is the smooth locus of $\mathcal C_\Pi$. We have the tautological morphism $\mathcal P_\Pi\longrightarrow \bbP^n$  that, by Theorem \ref {thm:basic}, is birational.    Let $\pi\cong \bbP^{n-1}$ be a general hyperplane in     $\bbP^n$. Then we have a rational map $U_\Pi \dasharrow \pi$,    which      sends a point of $U_\Pi$, corresponding to a Palatini line $\ell$, to the intersection of $\ell$    and     $\pi$. The rational map $U_\Pi \dasharrow \pi$ is birational by Theorem \ref {thm:basic}    and Lemma \ref{lem:proj} (with $W=\bbP^n$), and the assertion follows.

\emph{Case $n+1\leq m\leq \frac {n(n+1)}2-1$.} As in the previous case, the general Pfaffian hypersurface of degree $k+1$ in $\bbP^m$ is of the form $\mathcal C_\Pi$, where now $\Pi$ is a general linear subspace of codimension $c=m+1$ of $\bbP^{\frac {n(n+1)}2-1}$ (with $\Pi=\emptyset$ if $m=\frac {n(n+1)}2-1$), so that $\mathcal L_\Pi\cong \bbP^m$. Let $\pi\cong\bbP^{n-1}$ be a general hyperplane of $\bbP^n$, and let $I_\pi\subset \pi\times \mathcal C_\Pi$ be the closure of the set of pairs $(p,\omega)$ such that $p$ lies in the centre of $\omega$. If $p\in \pi$ is general, the complexes of $\mathcal L_\Pi$ whose centre contains $p$ are defined by $n$ independent linear conditions (this holds for a general linear subspace of dimension $n$ of $\mathcal L_\Pi$ by the proof of Theorem \ref{thm:basic}, hence also  for $\mathcal L_\Pi$), so that they form a linear subspace of $\mathcal L_\Pi$ of dimension $m-n\geq 1$. Hence, over a dense open subset of $\pi$, $I_\pi$ is a projective bundle with fibre $\bbP^{m-n}$, and therefore $I_\pi$ is rational. On the other hand, the projection $I_\pi\to \mathcal C_\Pi$ is birational: it is surjective, because the centre of any complex of $\mathcal C_\Pi$ meets $\pi$, and the centre of the general complex $\omega\in \mathcal C_\Pi$ is a line $\ell_\omega$ (see \S \ref{sserc:ls}), which meets $\pi$ in a single point. The claim follows.

\emph{Case $m\geq \frac {n(n+1)}2$.} The $\frac {n(n+1)}2$ entries $a_{ij}$, with $0\leq i<j\leq n$, of a general antisymmetric matrix of order $n+1$ of linear forms in $m+1$ variables are linearly independent. Hence we may take them as the first $\frac {n(n+1)}2$ homogeneous coordinates of $\bbP^m$, and the Pfaffian of the matrix does not depend on the remaining $m+1-\frac {n(n+1)}2\geq 1$ coordinates. Therefore the hypersurface it defines is a cone, with vertex a linear space of dimension $m-\frac {n(n+1)}2$, over the hypersurface $\mathcal C_n=\check {\mathbb G}(1,n)$ of $\check \bbP^{\frac {n(n+1)}2-1}$ defined by the Pfaffian of the generic antisymmetric matrix ${\bf A}$ of order $n+1$. Since $\check {\mathbb G}(1,n)$ is rational by the previous case with $m=\frac {n(n+1)}2-1$, and a cone over a rational variety is rational.
\end{proof}

Note that the first instance of this theorem, for $m=n=5$, says that the general Pfaffian cubic fourfold is rational, a fact proved already in \cite[Prop. 5] {BD} (see also \cite{Tr}). In fact, the proof of the case $m=n$ is a direct generalization of the argument of Beauville and Donagi, who show that, if $\mathcal C_\Pi$ is a general Pfaffian cubic fourfold and $\pi$ is a general hyperplane of $\bbP^5$, the map $\mathcal C_\Pi\dasharrow \pi$ sending a singular complex $\omega$ to the intersection of its centre $\ell_\omega$ with $\pi$ is birational (see also \cite[\S 1.6]{Has}).

\begin{remark}\label{rem:ratm} For $n+1\leq m\leq \frac {n(n+1)}2-1$, Theorem \ref{thm:rat} concerns only the Pfaffian hypersurface $\mathcal C_\Pi$, and not the linear section $X_\Pi=\Pi\cap \mathbb G(1,n)$ associated with the same $\Pi$. Indeed, in this range one has $c=m+1\geq n+2$ and $K_{X_\Pi}\sim \mathcal O_{X_\Pi}(c-n-1)$, with $c-n-1\geq 1$, so that $X_\Pi$ is of general type, or empty if $c>2n-2$. Hence, contrary to the case $c=n$, in which $X_\Pi$ is birational to $\mathcal C_\Pi$ (see Theorem \ref{thm:bir}), and to the case $c=n+1$, in which $X_\Pi$ is birational to a general hyperplane section of $B_\Pi$ (see Corollary \ref{cor:sec}), for $m\geq n+1$ the rationality of $\mathcal C_\Pi$ gives no information on $X_\Pi$.
\end{remark}

\begin{corollary}\label{cor:evenrat} Let $n=2k+1\geq 5$. Then the general linear section of the Grassmannian $\mathbb G(1,n+1)$ with a linear space of codimension $n+1$, which is a Fano variety of index 1 and dimension $n-1$ (see Section \ref{sec:even}), is rational. Hence the general Fano linear section of index 1 of $\mathbb G(1,N)$ is rational for every even $N\geq 4$.
\end{corollary}

\begin{proof} For $N=4$ the section in question is a Del Pezzo surface of degree 5, which is rational. Let $N=n+1$, with $n=2k+1\geq 5$, and let $\Pi'$ be a general linear subspace of codimension $N$ of $\bbP^{\frac {N(N+1)}2-1}$. We use the notation of Section \ref{sec:even}, with $N=2(k+1)$ in place of $n$: thus $\mathcal L_{\Pi'}\cong \bbP^{N-1}=\bbP^n$, and ${\bf {\mathcal A}}$ is an antisymmetric matrix of order $N+1$ of linear forms on $\mathcal L_{\Pi'}$. By Corollary \ref{cor:bir1} and its proof, $X_{\Pi'}$ is birational to a general hyperplane section $H=P_{\Pi'}\cap \Sigma$ of the Palatini locus $P_{\Pi'}\subset \bbP^N=\bbP({\bf V})$, where $\Sigma=\bbP({\bf W})$ is a general hyperplane. The map $\tau_{\Pi'}$ is birational and sends the general complex $\omega\in \mathcal L_{\Pi'}$ to its centre, which spans the kernel of ${\bf {\mathcal A}}$ at $\omega$ (see Theorem \ref{thm:det}). Hence $H$ is birational to the hypersurface of $\mathcal L_{\Pi'}$ formed by the complexes whose centre lies in $\Sigma$, i.e., whose restriction to ${\bf W}$ is degenerate. This is the hypersurface defined by the Pfaffian of the restriction of ${\bf {\mathcal A}}$ to ${\bf W}$, which is an antisymmetric matrix of order $N=n+1$ of linear forms on $\mathcal L_{\Pi'}\cong \bbP^n$. Since the restriction map $\wedge^2\check{\bf V}\to \wedge^2\check{\bf W}$ is surjective, by the generality of $\Pi'$ and $\Sigma$ this is a general Pfaffian hypersurface of degree $k+1$ in $\bbP^n$. The claim follows from Theorem \ref{thm:rat} with $m=n$. For $N=6$, i.e., $k=2$, the Pfaffian hypersurface in question is a general Pfaffian cubic fourfold, so in this case the assertion also follows directly from \cite[Prop. 5]{BD}.
\end{proof}

\subsection{The focal locus} \label{ssec:focal}    We now apply the theory of foci recalled in \S \ref{ssec:foci}    to the Palatini family  $\phi_\Pi: \mathcal P_\Pi\longrightarrow U_\Pi$, where $U_\Pi$ is the smooth locus of $\mathcal C_\Pi$. We will denote by $Z_\Pi$  the focal locus of the Palatini family.

\begin{lemma}\label{lem:inf} The focal locus $Z_\Pi$  of the Palatini family coincides with {the closure of} the set of points $p\in \bbP^n$ such that there are infinitely many Palatini lines passing through $p$. As a consequence $\dim(Z_\Pi)\leq n-2$.
\end{lemma} 

\begin{proof} {This follows from Theorem \ref{thm:basic} and Lemma \ref{lem:focinf}, applied to the Palatini family.} \end{proof}

We can be more precise about the focal locus $Z_\Pi$ arguing as follows. {Consider the maps $\alpha_H$, for $H\in \mathcal L_\Pi$, introduced in \S \ref{ssec:omega}.}

Now $\mathcal L_\Pi$ is spanned by $n+1$ (general) linearly independent hyperplanes $H_0,\ldots, H_n$, and therefore the datum of $\alpha_{H_0}, \ldots, \alpha_{H_n}$  gives an injective map {(i.e., the map $\alpha_{\mathcal L_\Pi}$ of \eqref{eq:alphaL})}
\begin{equation}\label{eq:alpha}
\alpha: T_{\bbP^n}(-2) \longrightarrow \mathcal O_{\bbP^n}^{n+1}.
\end{equation}

\begin{theorem}\label{thm:prec} One has $\dim(Z_\Pi)= n-2$ and 
$$
\deg(Z_\Pi)={n\choose 2}-1.
$$
The Palatini lines are {$(n-1)$--secant to} $Z_\Pi$ and through a general point of $Z_\Pi$ pass the lines of a family of dimension 1.
\end{theorem} 

\begin{proof} The focal locus $Z_\Pi$ is the degeneracy locus of $\alpha$. {Indeed, by \S\ref{ssec:omega}, the degeneracy locus $D(\alpha)$ of $\alpha$ is the closed set of points $p\in\bbP^n$ such that the linear system $\mathcal L_{\Pi,p}$ of the complexes of $\mathcal L_\Pi$ whose centre contains $p$ has positive dimension. If infinitely many Palatini lines pass through $p$, then $p\in D(\alpha)$, hence $Z_\Pi\subseteq D(\alpha)$ by Lemma \ref{lem:inf}. Conversely, if $p\in D(\alpha)$ and $\mathcal L_{\Pi,p}$ is not contained in ${\rm Sing}(\mathcal C_\Pi)$, the fibre over $p$ of the tautological morphism $\mathcal P_\Pi\to\bbP^n$ contains $(\mathcal L_{\Pi,p}\cap U_\Pi)\times\{p\}$, which is positive dimensional, so that $p\in Z_\Pi$ (see the proof of Lemma \ref{lem:focinf}). On the other hand, by the generality of $\Pi$, the complexes of $\mathcal L_\Pi$ whose centre has dimension $2j-1\geq 3$ form a family of dimension at most $n-j(2j-1)$, so that the points lying in the centre of some complex of ${\rm Sing}(\mathcal C_\Pi)$ form a set of dimension at most $n-3$. Since every irreducible component of $D(\alpha)$ has dimension at least $n-2$, it follows that $D(\alpha)\subseteq Z_\Pi$}. This implies that $\dim(Z_\Pi)\geq n-2$ {(note that $D(\alpha)$ is not empty, because its class, computed below by Porteous' formula, is non--zero)},    and thus      $\dim(Z_\Pi)= n-2$ by Lemma \ref {lem:inf}. As for the formula of the degree of $Z_\Pi$, this is    a straightforward application of Porteous' formula.     The final assertion is obvious. \end{proof}

\section{Base lines} \label{sec:base}

\subsection {The locus of base lines} \label{ssec:base}

In the Calabi--Yau case, the expected dimension of the locus $B_\Pi$ of base lines is $n-2$. More precisely, we have:

\begin{proposition} \label{prop:fin} Let $\ell$ be a general base line of $\mathcal L_\Pi$. Then $\ell$ does not meet any other base line of $\mathcal L_\Pi$. As a  consequence,   through the general point of $B_\Pi$ passes a unique base line of $\mathcal L_\Pi$ and the dimension of  the locus   $B_\Pi$ of base lines of $\mathcal L_\Pi$ is $n-2$. 
\end{proposition} 

\begin{proof}  Let $\ell$ be a general base line of $\mathcal L_\Pi$. The set $\mathbb I_\ell$ of points of $\mathbb G(1,n)$ corresponding  to    lines  incident     to $\ell$ is well known to be the intersection of $\mathbb G(1,n)$ with the tangent space $T_{\mathbb G(1,n),\ell}$.    Furthermore,     $\mathbb I_\ell$ is the cone, with vertex the point $[\ell]\in T_{\mathbb G(1,n),\ell}$ over the Segre variety ${\rm Seg}(1,n-2)\cong \bbP^1\times \bbP^{n-2}$,     Segre-embedded     in $\bbP^{2n-3}$. In  particular, the intersection of $\mathbb I_\ell$ with any  hyperplane in $T_{\mathbb G(1,n),\ell}$ not containing $[\ell]$ is ${\rm Seg}(1,n-2)$.

The intersection of $\mathbb I_\ell$ with $X_\Pi$ equals the intersection of $\mathbb I_\ell$ with $\Pi\cap T_{\mathbb G(1,n),\ell}$, where  $\dim (\Pi\cap T_{\mathbb G(1,n),\ell})=2n-2-(n+1)=n-3$. 
Since both $\mathbb I_\ell$ and $\Pi\cap T_{\mathbb G(1,n),\ell}$ are cones with vertex $[\ell]$, their intersection is also the cone, with vertex $[\ell]$, over the intersection of ${\rm Seg}(1,n-2)$ with a general hyperplane section of $\Pi\cap T_{\mathbb G(1,n),\ell}$,  inside     the ambient linear space of dimension $\dim (T_{\mathbb G(1,n),\ell})-1=2n-3$, {i.e.,} the span of ${\rm Seg}(1,n-2)$. By 
the  generality of $\Pi$, this intersection is clearly empty, and therefore $\mathbb I_\ell\cap X_\Pi=\{[\ell]\}$, as wanted. This proves the first assertion of the statement. The rest is an obvious consequence.
\end{proof}

\begin{corollary}\label{cor:sec} The Calabi--Yau variety $X_\Pi$ is birationally equivalent to a general hyperplane section of $B_\Pi$. 
\end{corollary}

\begin{proof} Let $\pi$ be a general hyperplane of $\bbP^n$ and let $B_{\Pi,\pi}$ be the section of $B_\Pi$ with $\pi$. There is a rational map 
$$
\rho_{\Pi,\pi}: X_\Pi\dasharrow B_{\Pi,\pi}
$$
that maps the general point $[\ell]\in X$ to the intersection of the line $\ell$ with $\pi$. By Proposition \ref {prop:fin}    and Lemma \ref{lem:proj}, this map is birational onto its image. 
\end{proof}

\begin{remark}\label{rem:sec} {By Lemma \ref{lem:projR} (with $c=n+1$),} the indeterminacy locus {$\mathcal R_{\Pi,\pi}$} of the birational map $\rho_{\Pi,\pi}$ is the set of  points  in $X_\Pi$ corresponding to lines contained in $\pi$
{, i.e.,} the intersection of the Grassmannian of lines in $\pi$, that  is a $\mathbb G(1,n-1)$, with $\Pi$. So  
$$
\dim ( {\mathcal R_{\Pi,\pi}})=n-5, \,\,\, \deg({\mathcal R_{\Pi,\pi}})=\deg (\mathbb G(1,n-1))=\frac {(2n-4)!}{(n-1)!(n-2)!},
$$
and {$\mathcal R_{\Pi,\pi}$} spans a linear space of dimension 
$$
\frac {n(n-1)}2-1-(n+1)=\frac {n(n-3)} 2-2. 
$$
In particular for $n=5$, {$\mathcal R_{\Pi,\pi}$} consists of 5 points spanning a $\bbP^3$. {Moreover} $B_{\Pi,\pi}\subset \bbP^{n-1}$ turns out to be the projection of $X_\Pi$ from the span of {$\mathcal R_{\Pi,\pi}$}.

For example, for $n=5$, $B_{\Pi,\pi}\subset \bbP^4$ is a surface of degree $9=14-5$, birational projection of the K3 surface $X_\Pi$ of degree 14 in $\bbP^ 8$ from a 5--secant $\bbP^3$ to $X_\Pi$, and $B_{\Pi,\pi}$ contains the 5 lines that are the exceptional divisors corresponding to the points from which we project. 
\end{remark}

\subsection{The locus of base lines and the focal locus} \label{ssec:basfoc}

In the case $n=5$ we verified that $B_\Pi$ and $Z_\Pi$ have the same degree 9. This is not by chance. Indeed, one has:

\begin{theorem}\label{thm:main} The locus of base lines $B_\Pi$ coincides with the focal locus $Z_\Pi$.
\end{theorem} 

\begin{proof} {Let $p\in Z_\Pi$. By Theorem \ref{thm:prec} and its proof, $p$ lies in the degeneracy locus of $\alpha$, i.e., the linear system $\mathcal L_{\Pi,p}$ of the complexes of $\mathcal L_\Pi$ whose centre contains $p$ has positive dimension. By Lemma \ref{lem:hyp}, with $\dim(\mathcal L_\Pi)=n$, the base lines of $\mathcal L_\Pi$ passing through $p$ are the lines through $p$ contained in the base locus of $\mathcal H_{\mathcal L_\Pi,p}$, which is a linear space of dimension $n-1-\dim(\mathcal H_{\mathcal L_\Pi,p})=\dim(\mathcal L_{\Pi,p})\geq 1$. In particular there are base lines through $p$, i.e., $p\in B_\Pi$. This proves that $Z_\Pi\subseteq B_\Pi$. By Theorem \ref{thm:prec} and Proposition \ref{prop:fin}, both $Z_\Pi$ and $B_\Pi$ have dimension $n-2$. Moreover $B_\Pi$ is irreducible, being the image of the $\bbP^1$--bundle $\mathcal X_\Pi$ over the irreducible variety $X_\Pi$. Hence $Z_\Pi=B_\Pi$.}
\end{proof}

{By Theorems \ref{thm:main} and \ref{thm:prec}, the Palatini lines are $(n-1)$--secant to $B_\Pi$. Moreover, one has the following:}

\begin{lemma}\label{lem:prel} The general Palatini line $\ell$ intersects exactly $n-1$ base lines. 
\end{lemma} 

\begin{proof} The proof is    similar to that     of Proposition \ref {prop:fin}. In fact, let $\ell$ be a general Palatini line. Then, recalling Remark  \ref {rem:ram1}, $T_{\mathbb G(1,n),\ell}$ and $\Pi$ span a hyperplane, hence they intersect in a linear space of dimension $n-2$. The set of lines $\mathbb I_\ell$ incident to $\ell$ 
is contained in $T_{\mathbb G(1,n),\ell}$ and it is 
the cone, with vertex $[\ell]\in T_{\mathbb G(1,n),\ell}$ over the already mentioned Segre variety ${\rm Seg}(1,n-2)$ in $\bbP^{2n-3}$, so it has dimension $n$ and codimension $n-2$ in $T_{\mathbb G(1,n),\ell}$. By the generality of $\Pi$, the linear space $\Pi\cap T_{\mathbb G(1,n),\ell}$ is a  general subspace  of codimension $n-2$ in $T_{\mathbb G(1,n),\ell}$ and therefore it intersects transversely $\mathbb I_\ell$ in as many distinct points as its degree, {   that is, } $n-1$. These points correspond to just $n-1$ lines in $X_\Pi$ that are incident to $\ell$.
\end{proof}

\begin{remark}\label{rem:DP}     For $n=5$, $X_\Pi$ is a K3 surface of degree 14 in $\bbP^8$ and $B_\Pi$ is the threefold of degree 9 in $\bbP^5$ swept out by its lines, which appears as case (4b) in the classification of threefolds of $\bbP^5$ with one apparent quadruple point given in \cite[Thm. 0.1]{DP}. By \cite[Lemma 4.2]{DP}, the 4--secant lines of $B_\Pi$ form an irreducible congruence of order one, which therefore coincides with the Palatini congruence by Lemma \ref{lem:prel}. Hence, for $n=5$, Theorems \ref{thm:basic}, \ref{thm:prec} and \ref{thm:main} overlap with the results in \cite{DP}; our proofs, valid for any odd $n$, are independent.    
\end{remark}

\subsection {Birational parametrization of the general Pfaffian hypersurface of degree $k+1$ in $\bbP^{2k+1}$}\label{ssec:param}

Let us go back to the map $\alpha$ defined in \eqref {eq:alpha}. {Arguing as in the proof of Lemma \ref{lem:degenv} (for every $p=[{\bf p}]\in \bbP^n$ the map $h\mapsto \alpha_H$, from $\wedge^2\check {\bf V}$ to the dual of the fibre of $T_{\bbP^n}(-2)$ at $p$, is surjective, because every linear form on $\bf V$ vanishing on $\bf p$ is of the form $h({\bf p},\cdot)$, with $h\in {\wedge^2\check {\bf V}}$), one sees that, for $\Pi$ general, the scheme $D(\alpha)$ defined by the maximal minors of $\alpha$ is Cohen--Macaulay of pure codimension 2 and smooth off a closed subset of codimension at least 6, hence it is reduced. Thus $D(\alpha)$ coincides with $Z_\Pi$ also as a scheme.} By Theorem \ref {thm:main} {and Lemma \ref{lem:HB} (recall that $Z_\Pi$ has codimension 2 by Theorem \ref{thm:prec})} we have an exact sequence
\begin{equation}\label{eq:exo}
0\longrightarrow T_{\bbP^n}(-2) \stackrel{\alpha} \longrightarrow \mathcal O_{\bbP^n}^{n+1}\longrightarrow \mathcal I_{Z_\Pi|\bbP^n}(n-1) \cong \mathcal I_{B_\Pi|\bbP^n}(n-1) \longrightarrow 0.
\end{equation}
Since $h^0(\bbP^n, T_{\bbP^n}(-2))=h^1(\bbP^n, T_{\bbP^n}(-2))=0$, we have $h^0(\bbP^n, \mathcal I_{B_\Pi|\bbP^n}(n-1))=n+1$. 

\begin{lemma}\label{lem:restr} Let $\pi\cong \bbP^{n-1}$ be a general hyperplane of $\bbP^n$.  Then 
$h^0({\bbP^{n-1}}, \mathcal I_{B_{\Pi,\pi}|\bbP^{n-1}}(n-1))=n+1$.
\end{lemma}

\begin{proof} First notice that $h^0(\bbP^n, \mathcal I_{B_\Pi|\bbP^n}(n-2))=0$. This follows from the exact sequence \eqref {eq:exo} twisted by $\mathcal O_{\bbP^n}(-1)$, noting that $h^1(\bbP^n, T_{\bbP^n}(-3))=0$.

From the exact sequence \eqref {eq:exo} twisted by $\mathcal O_{\bbP^n}(-1)$ we  deduce  that 
$$h^1(\bbP^n, \mathcal I_{B_\Pi|\bbP^n}(n-2))=0,$$ because $h^2(\bbP^n, T_{\bbP^n}(-3))=0$. 

Then from the exact sequence
$$
0\longrightarrow  \mathcal I_{B_\Pi|\bbP^n}(n-2) \longrightarrow  \mathcal I_{B_\Pi|\bbP^n}(n-1)\longrightarrow  \mathcal I_{B_{\Pi,\pi}|\bbP^{n-1}}(n-1) \longrightarrow 0
$$
we deduce that 
$h^0(\bbP^{n-1}, \mathcal I_{B_{\Pi,\pi}|\bbP^{n-1}}(n-1))=h^0(\bbP^n, \mathcal I_{B_\Pi|\bbP^n}(n-1))=n+1$, as wanted.\end{proof}

Consider the  map
$$
\varphi_{\Pi,\pi}: \pi\cong \bbP^{n-1}\dasharrow  \bbP^n
$$ 
{which is} defined by the linear system  $\mathcal B_{\Pi,\pi}:=|\mathcal I_{B_{\Pi,\pi}|\bbP^{n-1}}(n-1)|$. 

\begin{theorem}\label{thm:inv} The map $\varphi_{\Pi,\pi}$ is birational onto its image{, which is} the Pfaffian hypersurface $\mathcal C_\Pi$.
\end{theorem}

\begin{proof}
Let $f_0,\ldots, f_n\in H^0(\bbP^n,\mathcal I_{B_\Pi|\bbP^n}(n-1))$ be the images of the standard basis of $H^0(\bbP^n,\mathcal O_{\bbP^n}^{n+1})$ via the map in \eqref{eq:exo}. By the above, they form a basis of $H^0(\bbP^n,\mathcal I_{B_\Pi|\bbP^n}(n-1))$, so that the rational map $\varphi_\Pi:\bbP^n\dasharrow \bbP^n$ defined by the linear system $|\mathcal I_{B_\Pi|\bbP^n}(n-1)|$ is $p\mapsto [f_0(p),\ldots,f_n(p)]$. Let $p\in \bbP^n\setminus B_\Pi$. By Lemma \ref{lem:HB}, the linear form on the fibre of $\mathcal O_{\bbP^n}^{n+1}$ at $p$ with coefficients $f_0(p),\ldots, f_n(p)$ is non--zero and vanishes exactly on the image of the fibre of $\alpha$ at $p$. Since $\alpha=\alpha_{\mathcal L_\Pi}$ (see \eqref{eq:alphaL}), this means that the vector $(f_0(p),\ldots,f_n(p))$ spans the kernel of the linear map induced on the fibres at $p$ by the dual map $w_{\mathcal L_\Pi}$ (see \eqref{eq:wL}). By \S \ref{ssec:omega}, identifying the target of $\varphi_\Pi$ with $\mathcal L_\Pi$, $\varphi_\Pi(p)$ is then the unique complex of $\mathcal L_\Pi$ whose centre contains $p$ (it is unique because $p\notin B_\Pi=Z_\Pi$, which is the degeneracy locus of $\alpha$, see Theorems \ref{thm:main} and \ref{thm:prec}). If $p$ is general, by Theorem \ref{thm:basic} $\varphi_\Pi(p)$ is the point of $\mathcal C_\Pi$ corresponding to the unique Palatini line through $p$. Hence $\varphi_\Pi$ contracts the Palatini lines to points and its image is $\mathcal C_\Pi$. By the proof of Lemma \ref{lem:restr}, the restriction map $H^0(\bbP^n,\mathcal I_{B_\Pi|\bbP^n}(n-1))\to H^0(\bbP^{n-1},\mathcal I_{B_{\Pi,\pi}|\bbP^{n-1}}(n-1))$ is an isomorphism, so that $\varphi_{\Pi,\pi}$ is the restriction of $\varphi_\Pi$ to $\pi$. Therefore $\varphi_{\Pi,\pi}$ maps the general point $p\in \pi$ to the point of $\mathcal C_\Pi$ corresponding to the Palatini line through $p$, i.e., $\varphi_{\Pi,\pi}$ is the inverse of the birational map $\mathcal C_\Pi\dasharrow \pi$ considered in the proof of Theorem \ref{thm:rat}. This concludes the proof.
\end{proof}

{As we saw in the proof of Theorem \ref{thm:inv},} $\varphi_{\Pi,\pi}$ is the restriction to $\pi$ of the rational map
$$
\varphi_{\Pi}: \bbP^{n}\dasharrow  \bbP^n
$$ 
{which is} defined by the linear system  $\mathcal B_{\Pi}:=|\mathcal I_{B_{\Pi}|\bbP^{n}}(n-1)|$. This map contracts to points the Palatini lines and its image is again the Pfaffian hypersurface $\mathcal C_\Pi$.

\begin{remark}\label{rem:scroll} Given a base line $\ell$ corresponding to a point of $X_\Pi$, we can consider the set $\Sigma_\ell$ of all points $\omega\in \mathcal C_\Pi$ such that the {centre} of $\omega$ intersects $\ell$. One checks that $\Sigma_\ell$ is a rational normal scroll of degree $n-1$ in $\mathcal C_\Pi$. So  the Pfaffian hypersurface $\mathcal C_\Pi$ contains an $(n-3)$--dimensional family of such scrolls, parametrized by the Calabi--Yau variety $X_\Pi$. The first incarnation of this phenomenon is the 2--dimensional family of quartic scrolls contained in a Pfaffian cubic fourfold (see \cite{BD,BRS,K}), that are parametrized by the ``orthogonal'' degree 14 K3 surface, but this phenomenon appears in much greater generality. While the {appearance} of the K3 surfaces in the intermediate cohomology of a rational cubic fourfold is nowadays considered normal, in the higher dimensional cases this construction seems to point to a ``CY part'' inside the cohomology of Pfaffian rational hypersurfaces of degree $k+1$, given by these families of scrolls.
\end{remark}

\end{document}